%% file: paper2_ver.4.tex
\documentclass[11pt,reqno]{amsart}
\input{source/preamble}

\begin{document}
\input{source/abstract}

\input{source/Introduction.tex}
\input{source/Preliminary.tex}
\input{source/Period_Map_and_Period_Domain.tex}
\input{source/HK_Quot_and_Surj_of_Period_Map.tex}
\input{source/Examples.tex}
\input{source/Future_Directions.tex}
\input{source/Acknowledgments.tex}

\enlargethispage{1.5\baselineskip} 
\bibliography{source/references}
\bibliographystyle{jabbrv}

\end{document}

%% file: source/preamble.tex
\usepackage[T1]{fontenc}
  \usepackage[utf8]{inputenc}
\usepackage{geometry}
\usepackage{etoolbox}
\expandafter\preto\csname section\endcsname{\par\vspace{0.5\baselineskip}}
\expandafter\preto\csname subsection\endcsname{\par\vspace{0.3\baselineskip}}
\expandafter\preto\csname subsubsection\endcsname{\par\vspace{0.2\baselineskip}}

\numberwithin{equation}{section}

\usepackage{amsmath, amssymb, amsfonts, mathrsfs}
\usepackage{mathtools}
\usepackage{enumitem}

\usepackage{graphicx}
\usepackage[all]{xy}
\usepackage{extarrows}
\usepackage{color}

\usepackage{caption}
\usepackage{comment}

\usepackage{tabularx} 
\usepackage{booktabs}

\usepackage{hyperref}
\usepackage{bookmark}
\hypersetup{
    unicode            = true,
    bookmarksnumbered  = true,
    colorlinks         = false,
    hidelinks,
    final
}
\usepackage{xurl}
\usepackage{placeins}

\theoremstyle{plain}
\newtheorem{theorem}{Theorem}[section]
\newtheorem{lemma}[theorem]{Lemma}
\newtheorem{proposition}[theorem]{Proposition}
\newtheorem{corollary}[theorem]{Corollary}

\theoremstyle{definition}
\newtheorem{definition}[theorem]{Definition}
\newtheorem{example}[theorem]{Example}

\newtheorem{setup}[theorem]{Setup}
\newtheorem{assumption}[theorem]{Assumption}

\newtheorem{question}[theorem]{Question}
\newtheorem{problem}[theorem]{Problem}

\theoremstyle{remark}
\newtheorem{remark}[theorem]{Remark}



%% file: source/abstract.tex

\title{
A Torelli-type theorem for hyperk\"{a}hler quotients
}

\author{Ryota Kotani}
\address{Department of Mathematics, Institute of Science Tokyo}
\email{kotani.r.597b@m.isct.ac.jp, kotani.r.337@gmail.com}

\date{\today}

\begin{abstract}
In this paper, we investigate the moduli space of algebraic asymptotic hyperkähler structures on hyperkähler quotients arising from a quaternionic vector space $\mathbb{H}^n$.
We consider the period map on this moduli space, and prove a Torelli-type theorem (i.e., the bijectivity of the period map) under mild assumptions including the surjectivity of the Kirwan map.
This work provides an algebro-geometric generalization of the Torelli-type theorem for ALE gravitational instantons by Kronheimer (1989).
In particular, this theorem applies to toric hyperkähler varieties and Nakajima quiver varieties.
\end{abstract}

\maketitle

%% file: source/Introduction.tex
\section{Introduction}
\enlargethispage{1.5\baselineskip} 

In various fields of mathematics and theoretical physics, hyperkähler metrics asymptotic to a hyperkähler cone appear naturally as geometric objects linking moduli theory, representation theory, and algebraic geometry, and significant progress has been made in systematic research on their construction and classification. Typical examples of such metrics are constructed via hyperkähler quotients, including ALE gravitational instantons, toric hyperkähler varieties, and Nakajima quiver varieties.
\medskip

Regarding the classification of asymptotic hyperkähler metrics, the Torelli-type theorem for ALE gravitational instantons by Kronheimer \cite{Kronheimer89_HK_quot, Kronheimer89_Torelli-type} is well known. Here, a Torelli-type theorem asserts that for any triple of periods in the period domain, there exists a unique asymptotic hyperkähler structure realizing it (as will be explained later).
\medskip

Building on our previous work \cite{Kotani_26_PTM}, we first recall the basic setup for the general classification problem. Let $X$ be a conical symplectic variety (cf.~Def.~\ref{def: coni symp}), and assume that an algebraic hyperkähler cone metric $g_0$ (cf.~Def.~\ref{def:cone met on coni symp vrt}) is given on its regular locus $X_{\mathrm{reg}}$. Given any projective crepant resolution $Y$ of $X$, we define the moduli space $\mathcal{M}$ of algebraic hyperkähler structures on $Y$ asymptotic to the metric $g_0$ at infinity as follows:
\[
    \mathcal{M} \coloneqq \left\{(Y,g,I,J,K) \mid \text{$g$: algebraic hyperkähler metric asymptotic to $g_0$}\right\}/(\text{isomorphism}).
\]
Here, we say that the metric $g$ is \textit{asymptotic} to $g_0$ if the difference between the metrics $\|g-g_0\|_{g_0}$ on $X_{\mathrm{reg}} \subset Y$ converges to $0$ at infinity of $Y$ (cf.~Def.~\ref{def:asymp to cone metric}).
In addition, a hyperkähler metric $g$ is said to be \textit{algebraic} if its associated twistor space is algebraic (cf.~Def.~\ref{def:alg asymp HK on Y}). As we will see later, this algebraicity condition  plays a key role in establishing our main theorem.

\medskip
By applying the theory of the principal twistor model (PTM) constructed in our previous paper \cite{Kotani_26_PTM}, we deduce the following injectivity proposition on the moduli space.

\begin{proposition}[{Prop.~\ref{prop:inj of period map}}]\label{intro prop:inj of period map}
 The period map
 \[
 p \colon \mathcal{M} \to H^2(Y;\mathbb{R})\otimes\mathbb{R}^3;\quad (Y,g,I,J,K) \mapsto ([\omega_I],[\omega_J],[\omega_K]),
 \]
 which assigns each algebraic asymptotic hyperkähler structure to the triple of periods (i.e., cohomology classes) of the associated Kähler forms, is well-defined and injective.
\end{proposition}

\begin{remark}
  This injectivity implies that imposing the asymptotic condition ensures that the triple of periods uniquely determines the asymptotic hyperkähler metric.
\end{remark}

To describe the image of this period map (the set of realizable periods), we introduce the period domain $\Omega$, which is naturally determined by the universal Poisson deformation. Let $D$ be the wall (a finite union of hyperplanes) on $H^2(Y;\mathbb{C})$ arising as the discriminant locus of the universal Poisson deformation of $Y$ (cf.~Def.~\ref{def:wall}), and let its real part be $D_\mathbb{R}\subset H^2(Y;\mathbb{R})$. Then, we define the period domain as $\Omega \coloneqq (D_\mathbb{R}\otimes\mathbb{R}^3)^c \subset H^2(Y;\mathbb{R})\otimes\mathbb{R}^3$ (cf.~Def.~\ref{def:period domain}).
By the non-degeneracy of the metrics, the following proposition holds.

\begin{proposition}[{Prop.~\ref{prop:inclusion to period domain}}]\label{intro prop:inclusion}
 The image of the period map $p$ is contained in the period domain $\Omega$. That is, $p(\mathcal{M}) \subset \Omega$ holds.
\end{proposition}

Consequently, to understand the moduli space, it remains to prove the surjectivity of the period map. Namely, this addresses the question: for any triple of periods in the period domain, does an asymptotic hyperkähler structure realizing it actually exist? We show that this surjectivity holds in the algebro-geometric setting for conical symplectic varieties constructed via hyperkähler quotients. The main theorem of this paper is the following Torelli-type theorem.

\begin{theorem}[{Theorem~\ref{main thm:Torelli-type thm for HK quot}}]\label{intro main thm}
  Let $G \subset Sp(n)$ be a compact Lie group acting linearly on $M\coloneqq \mathbb{H}^n$, and let $\mu \colon M \to \mathfrak{g}^* \otimes \mathbb{R}^3$ be the standard hyperkähler moment map associated with this action (cf.~Setup~\ref{setup:HK quot}).
  Let $X_0 \coloneqq \mu^{-1}(0,0,0)/G$ be the central quotient, which naturally forms a conical symplectic variety, and let $g_0$ be the natural hyperkähler cone metric on its regular locus $(X_0)_{\mathrm{reg}}$.
  Assume the following two conditions:
  \begin{enumerate}
      \item The stable locus $\mu^{-1}_\mathbb{C}(0)^s$ is non-empty.
      \item There exists a generic lattice point $\beta\in\mathfrak{z}^*_\mathbb{Z}$ such that the group $G_\mathbb{C}$ acts freely on the stable locus $M^{\beta-s}$, and the Kirwan map $\kappa_{\beta}$ is surjective.
  \end{enumerate}

  Then, the hyperkähler quotient $Y_{\beta}\coloneqq\mu^{-1}(\beta,0,0)/G \simeq\mu_\mathbb{C}^{-1}(0)/\!/_{\beta} G_\mathbb{C}$ is a projective crepant resolution of $X_0$.
  Furthermore, for the moduli space $\mathcal{M}$ of algebraic hyperkähler structures on $Y_{\beta}$ asymptotic to $g_0$ at infinity, the period map $p \colon \mathcal{M} \to \Omega$ is bijective.
\end{theorem}

\begin{remark}\label{intro_remark:meaning of assumption}
  The condition $\mu_\mathbb{C}^{-1}(0)^s \neq \emptyset$ provides a sufficient condition for the natural map $Y_{\beta}\to X_0$ to be birational by a standard GIT argument, while the free action of $G_\mathbb{C}$ ensures that $Y_{\beta}$ is a non-singular variety. Moreover, the surjectivity of the Kirwan map $\kappa_{\beta}$ is a condition to guarantee the surjectivity of the period map $p$. For a precise summary of these roles, see Remark~\ref{rem:roles_of_assumptions}.
\end{remark}

To clarify the relationship with prior work, Table~\ref{tab:comparison} compares the setting of this paper with the foundational results of Kronheimer \cite{Kronheimer89_HK_quot, Kronheimer89_Torelli-type} on the moduli of ALE gravitational instantons. In this table, ``Injectivity'' and ``Surjectivity'' refer to the corresponding properties of the period map $p$ in establishing the Torelli-type theorem.

\begin{table}[htbp]
  \centering
  \caption{Comparison with prior work}
  \label{tab:comparison}
  \begin{tabularx}{\textwidth}{@{} p{2.2cm} X X @{}}
    \toprule
     & Kronheimer (1989) \cite{Kronheimer89_HK_quot, Kronheimer89_Torelli-type} & This paper \\
    \midrule
    Underlying spaces & Quotient space $X=\mathbb{C}^2/\Gamma \ (\Gamma \subset SU(2))$ and its minimal resolution $Y$ & Central quotient $X_0$ and its projective crepant resolution $Y_{\beta}$ constructed via hyperkähler quotients \\
    \addlinespace
    Asymptotic condition & ALE condition (higher-order asymptotic conditions including covariant derivatives of the metric) & Weak asymptotic condition (the difference between the metrics $\|g-g_0\|_{g_0}$ converges to $0$ at infinity of $Y$) \\
    \addlinespace
    Metrics & ALE hyperkähler metrics on $Y$ & Algebraic asymptotic hyperkähler metrics on $Y_{\beta}$ \\
    \addlinespace
    Injectivity & Compactification of twistor spaces \cite{Kronheimer89_Torelli-type} & Universality of the principal twistor model \cite{Kotani_26_PTM} \\
    \addlinespace
    Surjectivity & Construction of hyperkähler quotients and identification of periods via deformations \cite{Kronheimer89_HK_quot} & Construction of hyperkähler quotients and identification of periods via Poisson deformations \\
    \bottomrule
  \end{tabularx}
\end{table}

As shown in Table~\ref{tab:comparison}, Kronheimer \cite{Kronheimer89_HK_quot,Kronheimer89_Torelli-type} proved the Torelli-type theorem for ALE gravitational instantons under an analytic setting. In contrast, this work establishes a broader framework for algebraic asymptotic hyperkähler metrics. Crucially, this algebro-geometric setting inherently encompasses the ALE gravitational instantons (as we will see in \S\ref{subsec:ALE-grav inst}). In this sense, the results of this paper can be regarded as a natural algebro-geometric generalization of Kronheimer's foundational work. Furthermore, this generalized Torelli-type theorem applies to a much wider class of algebraic varieties, including toric hyperkähler varieties and Nakajima quiver varieties, provided they satisfy the assumptions of Theorem~\ref{intro main thm} (cf.~\S\ref{sec:examples}).\medskip

The paper is organized as follows:

In \S \ref{sec:preliminaries}, we prepare the geometric concepts for the subsequent sections. First, we recall conical symplectic varieties and their Poisson deformations, and in particular, we characterize the wall of the universal Poisson deformation. Next, we recall the basic setup for hyperkähler metrics asymptotic to a cone metric at infinity, and overview the theory of the principal twistor model \cite{Kotani_26_PTM}, which controls the moduli space of algebraic asymptotic hyperkähler structures.

In \S \ref{sec:period map and period domain}, we formulate the period map for the moduli space of asymptotic hyperkähler structures in a general setting. We explain that the injectivity of the period map is derived from the universality theorem for the principal twistor model (Proposition~\ref{intro prop:inj of period map}), and further prove that the image of the period map is contained in the period domain (Proposition~\ref{intro prop:inclusion}) by the characterization of the wall for the universal Poisson deformation.

In \S \ref{sec:HK quot and surj of the period map}, focusing on hyperkähler quotients, we prove the surjectivity of the period map, which is the main theorem (Theorem~\ref{intro main thm}). Specifically, we prepare an algebro-geometric description using GIT quotients via the Kempf--Ness type theorem for affine varieties \cite{King94}, and we clarify the relationship between hyperkähler quotients and Poisson deformations. We then analyze the asymptotic behavior of the metrics on the hyperkähler quotient, and prove the surjectivity by utilizing the correspondence between the periods of the Poisson deformation and those of the hyperkähler quotient.
Additionally, we verify the algebraicity of these metrics by investigating their associated twistor spaces.

In \S \ref{sec:examples}, we introduce toric hyperkähler varieties $X_\zeta(A)$ determined by a coloop-free unimodular matrix $A$, and Nakajima quiver varieties $X_\zeta(v,w)$ whose pair of dimension vectors $(v,w)$ forms a strict Schur root, as examples to which the main theorem is applicable. We also revisit the example of ALE gravitational instantons studied by \cite{Kronheimer89_HK_quot} to see how it fits into our framework.

Finally, we conclude the paper in \S \ref{sec:future directions} by discussing future directions, including the extension of the Torelli-type theorem to more general spaces.

%% file: source/Preliminary.tex
\section{Preliminaries} \label{sec:preliminaries}

In this section, we prepare the geometric concepts for the subsequent sections. First, we recall conical symplectic varieties and their Poisson deformations, and in particular, we characterize the wall of the universal Poisson deformation. Next, we recall the basic setup for hyperkähler metrics asymptotic to a cone metric at infinity, and overview the theory of the principal twistor model \cite{Kotani_26_PTM}, which controls the moduli space of algebraic asymptotic hyperkähler structures.

\medskip
First, we recall the definition of a conical symplectic variety.

\begin{definition} \label{def: coni symp} 
    We call an affine holomorphic symplectic variety $(X, \omega)$ equipped with a good $\mathbb{C}^*$-action $\lambda$ a \textit{conical symplectic variety}. Here, a \textit{good $\mathbb{C}^*$-action} is a $\mathbb{C}^*$-action $\lambda$ satisfying the following two conditions:
    \begin{enumerate}
     \item The coordinate ring $\Gamma(X,\mathcal{O}_X)$ is positively graded with respect to the action $\lambda$.
     \item The holomorphic symplectic $2$-form $\omega$ has positive weight with respect to the action $\lambda$.
    \end{enumerate}
\end{definition}

\begin{remark}
    A conical symplectic variety $X$ may be singular. In this case, we consider $\omega$ as a $2$-form on the regular locus $X_{\mathrm{reg}}$.
\end{remark}

\subsection{Poisson Deformations}\label{subsec:PD}

In this subsection, we briefly recall the universal Poisson deformation of a crepant resolution of a conical symplectic variety, and clarify the relationship between its discriminant locus and the wall in the cohomology group. This wall plays an important role later in determining the boundary of the period domain.

\medskip
We first recall the notion of Poisson deformations. A Poisson deformation of a holomorphic symplectic variety is a deformation of the complex variety equipped with a compatible deformation of the Poisson structure, which corresponds to a relative holomorphic symplectic form. A universal Poisson deformation is characterized by its universal property: any Poisson deformation can be obtained uniquely as a pullback from this family.

\medskip
Next, we recall a theorem of Namikawa on the universal Poisson deformations:

\begin{theorem}[{Namikawa \cite[\S 5]{Namikawa11}}] \label{Namikawa Thm:CM and univ Poisson}
    Let $X$ be a conical symplectic variety.
    Assume that $X$ admits a projective crepant resolution $Y$. Then, there exist universal Poisson deformations $\mathcal{X}\to\mathcal{B}$ and $\mathcal{Y}\to\mathcal{C}$ of $X$ and $Y$, respectively. Here, the base spaces $\mathcal{B}$ and $\mathcal{C}$ are vector spaces, and $\mathcal{C}$ is naturally isomorphic to $H^2(Y;\mathbb{C})$ via the period map.
    Furthermore, there exists a finite Galois cover $q \colon \mathcal{C}\to\mathcal{B}$ such that if we let $\mathcal{X}'\to\mathcal{C}$ be the pullback of $\mathcal{X}$ via $q$, then $\mathcal{Y}$ is a simultaneous resolution of $\mathcal{X}'$.
    That is, there exists a morphism $\nu \colon \mathcal{Y}\to\mathcal{X}'$ such that the following diagram commutes:
    \begin{equation}
    \vcenter{
    \xymatrix{
       \mathcal{Y} \ar[d]_{} \ar[r]^\nu &\mathcal{X}' \ar[d]^{}\ar[r] &\mathcal{X} \ar[d]^{} \\
       \mathcal{C} \ar@{=}[r] &\mathcal{C} \ar[r]^{q}&\mathcal{B} \\
    }} \label{diagram:Yuni, Xuni}
    \end{equation}
    Moreover, for each $c \in \mathcal{C}$, the restriction of the morphism $\nu$ to the fiber over $c$,
    \[
    \nu_c \colon \mathcal{Y}_{c}\to \mathcal{X}'_{c}=\mathcal{X}_{q(c)},
    \]
    is a projective crepant resolution of $\mathcal{X}'_{c}$, and for a generic $c\in\mathcal{C}$, $\nu_c$ is an isomorphism.
\end{theorem}

\begin{remark}
 We refer to the Poisson deformation $\mathcal{X}'$ as the \textit{affinization} of the universal Poisson deformation $\mathcal{Y}$ of $Y$.
\end{remark}

It is known that the good $\mathbb{C}^*$-action on a conical symplectic variety $X$ naturally extends to the universal Poisson deformation $\mathcal{Y}$ of its crepant resolution $Y$ (cf.~\cite{Namikawa08}).
In general, any Poisson deformation equipped with such a $\mathbb{C}^*$-action admits a smooth trivialization. Via such a trivialization, the period map of a Poisson deformation is defined as follows.

\begin{definition}\label{def: period map for PD}
  Let $X$ be a conical symplectic variety. Assume that $X$ admits a projective crepant resolution $Y$, and let $\mathcal{Y}\to T$ be a Poisson deformation of $Y$.
  Using a smooth trivialization $\mathcal{Y} \simeq Y \times T$, the \textit{period map} $\mathcal{P}$ associated with this deformation $\mathcal{Y}\to T$ is defined as the map
  \[
    \mathcal{P} \colon T \to H^2(Y;\mathbb{C}); \quad t \mapsto [\Omega_t],
  \]
  where $[\Omega_t]$ denotes the de Rham cohomology class of the holomorphic symplectic form $\Omega_t$ on the fiber $\mathcal{Y}_t \simeq Y$.
\end{definition}

\begin{remark}\label{rem:univ PD and period map}
  In terms of the period map, the universal Poisson deformation is characterized as follows:
  A Poisson deformation $\mathcal{Y}\to\mathcal{C}$ is universal if and only if its period map $\mathcal{C} \to H^2(Y;\mathbb{C})$ is an isomorphism.
\end{remark}

The discriminant locus of a Poisson deformation is defined as follows.
\begin{definition}
 Let $X$ be a conical symplectic variety. Assume that $X$ admits a projective crepant resolution $Y$, and let $\mathcal{Y}\to T$ be a Poisson deformation of $Y$.
 The discriminant locus $D'$ of the Poisson deformation $\mathcal{Y}\to T$ is the subset of $T$ defined by:
 \[D' \coloneqq \{t\in T \mid \text{the fiber }\mathcal{Y}_t \text{ is non-affine}\}.\]
 Equivalently, using the affinization $\mathcal{X}'\to T$ of $\mathcal{Y}$, it can be written as:
 \[D' \coloneqq \{t\in T \mid \text{the fiber }\mathcal{X}'_t \text{ is singular}\}.\]
\end{definition}

In the following, we characterize the discriminant locus $D'$ of the universal Poisson deformation. For this purpose, we use the following theorem on crepant resolutions.

\begin{theorem}[{Kaledin \cite[Theorem 1.9]{Kaledin_06}}]\label{thm:Kaledin}
 Let $X$ be a conical symplectic variety, and let $\pi \colon Y\to X$ be a projective crepant resolution (or equivalently, a symplectic resolution).
 Then, for every fiber $F$ of $\pi$, the even cohomology groups $H^{2p}(F;\mathbb{C})$ carry a pure $\mathbb{R}$-Hodge structure
 of weight $2p$ and type $(p,p)$ (and the odd cohomology groups vanish).
\end{theorem}

\begin{remark}
 This theorem implies that the fiber $F$, which is a possibly singular projective variety, admits a Hodge decomposition similar to that of a non-singular variety.
\end{remark}

We prove the following lemma.

\begin{lemma}\label{lemma:generator of H_2(Y)}
 Let $X$ be a conical symplectic variety, and let $\pi \colon Y\to X$ be a projective crepant resolution. Consider the (possibly singular) projective variety $Y_0 \coloneqq \pi^{-1}(o)\subset Y$. Here, $o\in X$ is the unique fixed point of the $\mathbb{C}^*$-action. Then, $H_2(Y;\mathbb{C})$ is generated by algebraic curves on $Y_0$.
\end{lemma}

\begin{proof}
 Since $Y$ retracts to $Y_0$ by the $\mathbb{C}^*$-action, we have $H^2(Y;\mathbb{C})\simeq H^2(Y_0;\mathbb{C})$.
 By the universal coefficient theorem, we have $H_2(Y_0;\mathbb{C})\simeq \mathrm{Hom}(H^2(Y_0;\mathbb{C}),\mathbb{C})$, and $H_2(Y_0;\mathbb{C})$ admits a pure $\mathbb{R}$-Hodge structure of weight $-2$.

 \medskip
 Let $f \colon \tilde{Y}_0 \to Y_0$ be a resolution of singularities of $Y_0$.
 Then, by Deligne's theory of mixed Hodge structures, the image of $f_* \colon H_2(\tilde{Y}_0;\mathbb{C})\to H_2(Y_0;\mathbb{C})$ coincides with the lowest weight piece $W_{-2}H_2(Y_0;\mathbb{C})$ of the weight filtration. Now, since $H_2(Y_0;\mathbb{C})$ is pure of weight $-2$ by Kaledin's theorem, $f_*$ is surjective.

 \medskip
 Since $\tilde{Y}_0$ is a non-singular projective variety, the Lefschetz $(1,1)$-theorem implies that $H_2(\tilde{Y}_0;\mathbb{C})$ is generated by algebraic cycles. Since the images of these generators under $f_*$ are algebraic cycles, $H_2(Y_0;\mathbb{C})$ is also generated by algebraic cycles.
 Consequently, the assertion follows.
\end{proof}

We now define the wall arising naturally from the universal Poisson deformation.

\begin{definition}\label{def:wall}
 Let $X$ be a conical symplectic variety, and let $\pi \colon Y\to X$ be a projective crepant resolution.
 By Lemma~\ref{lemma:generator of H_2(Y)}, we can write
 \[H_2(Y;\mathbb{C}) = \langle [\Sigma_\alpha]\mid\Sigma_\alpha \subset Y_0 :\text{algebraic curve} \rangle.\]
 Let $\mathcal{Y}\to\mathcal{C}$ be the universal Poisson deformation of $Y$, and consider the isomorphism $\mathcal{C}\simeq H^2(Y;\mathbb{C})$ via the period map (cf.~Remark~\ref{rem:univ PD and period map}).
 For each $\alpha$, we define
 \[H_\alpha \coloneqq \{[\Omega_c]\in H^2(Y;\mathbb{C})\mid \langle\Sigma_\alpha, \Omega_c\rangle \coloneqq \int_{\Sigma_\alpha}\Omega_c = 0\}.\]
 Here, $\Omega_c$ is the holomorphic symplectic $2$-form on the fiber $\mathcal{Y}_c$.
 Then, the wall $D$ on $H^2(Y;\mathbb{C})$ is defined by
 \[D \coloneqq \bigcup_\alpha H_\alpha.\]
\end{definition}

We prove the following proposition.
\begin{proposition}\label{prop:meaning of wall}
 Consider the setup of Definition~\ref{def:wall}, and let $D$ be the wall in $H^2(Y;\mathbb{C})$.
 Let $D'$ be the discriminant locus on the base space $\mathcal{C}$ of the universal Poisson deformation $\mathcal{Y}$.
 Then, under the isomorphism $\mathcal{C} \simeq H^2(Y;\mathbb{C})$ induced by the period map, $D'$ coincides with $D$.
\end{proposition}

\begin{proof}
 Since each fiber $\mathcal{Y}_c$ of the universal Poisson deformation $\mathcal{Y}\to\mathcal{C}$ is diffeomorphic to $Y$ as a smooth manifold, we can write $\mathcal{Y}_c= (Y,I_c)$ using the complex structure $I_c$ on $Y$.

 \medskip
 Assume that $c\in D'$, that is, $\mathcal{Y}_c= (Y,I_c)$ is non-affine.
 In this case, the target space $\mathcal{X}'_c$ of the affinization map $\nu_c \colon \mathcal{Y}_c \to \mathcal{X}'_c$ is singular. For a singular point $x \in \mathcal{X}'_c$, its fiber $F \coloneqq \nu_c^{-1}(x)\subset \mathcal{Y}_c$ is an $I_c$-holomorphic subvariety.
 Since $F$ is a projective variety, by taking successive hyperplane sections if necessary, we obtain an algebraic curve $\Sigma \subset F$ such that $\langle \Sigma, \Omega_c\rangle=0$.
 By Lemma~\ref{lemma:generator of H_2(Y)}, $H_2(Y;\mathbb{C})$ is generated by algebraic curves on $\mathcal{Y}_0\simeq Y$. Thus, there exists an $I_0$-holomorphic algebraic curve $\Sigma_{\alpha}$ on $Y$ such that $\langle \Sigma_\alpha, \Omega_c\rangle=0$. Therefore, $[\Omega_c] \in H_\alpha \subset D$.

 \medskip
 Conversely, assume that $[\Omega_c] \in D$.
 By definition, there exists an algebraic curve $\Sigma_\alpha$ on $Y$ such that $\langle\Sigma_\alpha, \Omega_c\rangle=0$ holds for the holomorphic symplectic $2$-form $\Omega_c$ on $\mathcal{Y}_c=(Y,I_c)$.
 Noting that $\Sigma_{\alpha}$ is a real cycle, $\langle\Sigma_\alpha, \bar{\Omega}_c\rangle=0$ also holds. This implies that the curve $\Sigma_\alpha$ is of $(1,1)$-type with respect to the complex structure $I_c$. Namely, the curve $\Sigma_\alpha$ is a holomorphic curve on $\mathcal{Y}_c$. Therefore, $\mathcal{Y}_c$ is non-affine. That is, $c\in D'$.
 Consequently, the assertion follows.
\end{proof}

\subsection{Setup for asymptotic hyperkähler metrics}\label{subsec:setup for AHK}

In this subsection, based on \cite{Kotani_26_PTM}, we define asymptotic hyperkähler metrics and the moduli space $\mathcal{M}$ studied in this paper.

\medskip
We first recall the definition of a cone metric on a conical symplectic variety.
\begin{definition}\label{def:cone met on coni symp vrt}
   Let $X$ be a conical symplectic variety with a $\mathbb{C}^*$-action $\lambda$.
   We say that a metric $g_0$ on the regular locus $X_{\mathrm{reg}}$ is a \textit{cone metric} if there exists a positive integer $k$ such that $\lambda_r^*g_0 = r^k g_0$ for any action $\lambda_r$ with $r>0$.
\end{definition}

\begin{remark}
    If the cone metric $g_0$ is a hyperkähler metric, the discussion on the twistor space in \cite{Kotani_26_PTM} automatically implies that weight $k=2$.
\end{remark}

When the conical symplectic variety $X$ admits a projective crepant resolution $Y$, the asymptotic behavior of a metric on the resolution $Y$ is defined as follows.

\begin{definition}\label{def:asymp to cone metric}
   Let $X$ be a conical symplectic variety. Assume that the regular locus $X_{\mathrm{reg}}$ admits a cone metric $g_0$.
   Let $Y$ be a projective crepant resolution of $X$.
   We say that a metric $g$ on the resolution $Y$ is \textit{asymptotic to the metric $g_0$ at infinity on $Y$} if it satisfies the following condition:
   Identifying the regular locus $X_{\mathrm{reg}}$ with a subset of $Y$ naturally, fix an arbitrary point $x_0\in X_{\mathrm{reg}}$. Then, we have
   \[\lim_{r\to \infty} \|g-g_0\|_{g_0(r\cdot x_0)}=0, \]
   where $r\cdot x_0$ denotes the action of $r>0$ on $X$, and $\|\cdot\|_{g_0(r\cdot x_0)}$ denotes the norm induced by the metric $g_0$ at the point $r\cdot x_0\in X_{\mathrm{reg}}$.
\end{definition}

\begin{remark}
    No assumption on the rate of convergence is required.
\end{remark}

\begin{example} Examples of algebraic asymptotic hyperkähler metrics are as follows:
 \begin{enumerate}[itemsep=\medskipamount]
   \item ALE gravitational instantons~(cf.~\cite{Kronheimer89_HK_quot}): Metrics asymptotic at infinity to the standard flat metric on the quotient space $X = \mathbb{C}^2 / \Gamma \ (\Gamma \subset SU(2))$ with a Kleinian singularity (cf.~\S\ref{sec:examples}).
   \item QALE hyperkähler metrics: Higher-dimensional analogs of ALE metrics, corresponding to the case $X = \mathbb{C}^{2n} / G \ (G \subset Sp(n))$, provided the metric is assumed to be algebraic.
   \item Nakajima quiver varieties~(cf.~\cite{Nakajima94}) and toric hyperkähler varieties~(cf.~\cite{Bielawski-Dancer00}): Metrics constructed via hyperkähler quotients (cf.~\S\ref{sec:examples}).
 \end{enumerate}
 \medskip
\end{example}

Throughout this and the subsequent section, we work under the following setup for the moduli space of algebraic asymptotic hyperkähler structures.

\begin{setup}\label{setup:general case}
 Let $X$ be a conical symplectic variety such that the regular locus $X_{\mathrm{reg}}$ admits an algebraic hyperkähler cone metric $g_0$, and let $Y$ be a projective crepant resolution of $X$.
 We define the moduli space $\mathcal{M}$ of algebraic asymptotic hyperkähler structures on $Y$ as follows:
 \[
     \mathcal{M} \coloneqq \left\{(Y,g,I,J,K) \mid \text{$g$: algebraic hyperkähler metric asymptotic to $g_0$}\right\}/(\text{isomorphism}).
 \]
\end{setup}
The precise technical meaning of an ``algebraic asymptotic hyperkähler structure'' will be reviewed in Definition \ref{def:alg asymp HK on Y} of the next subsection using the twistor framework.

\subsection{The principal twistor model}\label{subsec:PTM}\quad
The principal twistor model (PTM) \cite{Kotani_26_PTM} is a geometric framework developed to systematically study algebraic asymptotic hyperkähler structures by formulating the associated twistor spaces from the perspective of Poisson deformations.
In this subsection, we first formulate the precise notion of an algebraic asymptotic hyperkähler structure via the associated twistor space. We then recall the definition and universality of the PTM, which provides the foundation for establishing the injectivity of the period map.

\medskip
We first briefly recall the notion of twistor spaces.
By the twistor correspondence \cite{HKLR}, a hyperkähler structure $(M,g,I,J,K)$ corresponds one-to-one to a \textit{twistor space}, which is a tuple $(Z \to \mathbb{P}^1, \tau, \omega, \{\ell_x\}_{x\in M})$ consisting of a complex manifold $Z$ fibered over $\mathbb{P}^1$, a real structure $\tau$, an $\mathcal{O}(2)$-twisted relative holomorphic symplectic $2$-form $\omega$, and a family of real sections $\{\ell_x\}_{x\in M}$ foliating $Z$, called \textit{twistor lines}.
The tuple $(Z \to \mathbb{P}^1, \tau, \omega)$ obtained by forgetting the family of twistor lines $\{\ell_x\}_{x\in M}$ from the twistor space is called a \textit{twistor model}.
We say a hyperkähler structure is \textit{algebraic} if its twistor model is algebraic, meaning $Z$ can be expressed by gluing two copies of a Poisson deformation over $\mathbb{C}$ via an algebraic map, and $\tau$ is an algebraic morphism (cf.~\cite[Def.~3.29]{Kotani_26_PTM}).\medskip


Via the associated twistor space, an algebraic asymptotic hyperkähler structure on $Y$ is defined as follows.

\begin{definition}\label{def:alg asymp HK on Y}
Let $X$ be a conical symplectic variety, $g_0$ be a cone metric on $X_{\mathrm{reg}}$, and $Y$ be a projective crepant resolution of $X$.
We say that a hyperkähler structure $(Y,g,I,J,K)$ is an \textit{algebraic asymptotic hyperkähler structure} on $Y$ (with respect to $g_0$) if it satisfies the following three conditions:
\begin{enumerate}
    \item The hyperkähler metric $g$ is asymptotic to the cone metric $g_0$ at infinity on $Y$ in the sense of Definition \ref{def:asymp to cone metric}.
    \item The corresponding twistor space $Z \to \mathbb{P}^1$ is algebraic as a twistor model.
    \item For some point $u \in \mathbb{P}^1$, the fiber $Z_u$ of $Z$ at $u$ is connected to $Y$ by a Poisson deformation; that is, $Z_u$ is isomorphic, as a holomorphic symplectic variety, to a fiber $\mathcal{Y}_t$ of some Poisson deformation $\mathcal{Y} \to T$ of $Y$.
\end{enumerate}
\end{definition}

\begin{remark}
    The third technical condition enables us to apply algebraic Poisson deformation theory to the twistor space $Z$.
\end{remark}


By generalizing the notion of a twistor model, a \textit{principal twistor model} (PTM) is defined as follows.

\begin{definition}
Let $P$ be a $(2n+d+1)$-dimensional complex manifold ($n \ge 1, d \ge 0$). We call a tuple $(\varphi \colon P\to\mathcal{C}(2), \omega_P, \tau_P)$ satisfying the following (P1)--(P3) a \textit{principal twistor model} (PTM).
\begin{itemize}
    \item[(P1)] It admits a holomorphic fibration structure $\varphi \colon P \to \mathcal{C}(2)$, where the base space $\mathcal{C}(2) \coloneqq \mathcal{C} \otimes \mathcal{O}(2)$ is a vector bundle over $\mathbb{P}^1$ for a $d$-dimensional complex vector space $\mathcal{C}$.
    \item[(P2)] It admits a real structure $\tau_P \colon P \to \overline{P}$, which induces a real structure $\sigma_2 \colon \mathcal{C}(2) \to \overline{\mathcal{C}(2)}$ on the base space.
    \item[(P3)] There exists a relative holomorphic symplectic $2$-form $\omega_P \in \Gamma(P, \Omega^2_{\varphi} \otimes \pi_P^*\mathcal{O}(2))$ satisfying $\tau_P^*\omega_P = \overline{\omega_P}$, where $\pi_P \colon P \to \mathbb{P}^1$ is the natural projection.
\end{itemize}
\end{definition}

\begin{remark}
    When $d=0$, the principal twistor model coincides with a twistor model.
\end{remark}

A fundamental property of a PTM is given by the following proposition.

\begin{proposition}\label{prop:pullback twistor model}
    Let $(\varphi \colon P\to\mathcal{C}(2), \omega_P, \tau_P)$ be a PTM.
    For any real section $s \in H^0(\mathbb{P}^1, \mathcal{C}(2))^{\sigma_2}$ on the vector bundle $\mathcal{C}(2)$ (i.e., $\sigma_2 \circ s=s \circ \sigma_{\mathrm{ap}}$), consider the pullback $\pi_s \coloneqq s^*\varphi \colon Z_s \to \mathbb{P}^1$. Then, together with the restriction of the form $\omega_s \coloneqq \omega_P|_{Z_s}$ and the real structure $\tau_s \coloneqq \tau_P|_{Z_s}$, the tuple $(Z_s, \tau_s, \omega_s)$ forms a twistor model.
\end{proposition}

We state the universality theorem for the PTM $\mathcal{Y}(1)$.

\begin{theorem}[{\cite[Theorem~3.43]{Kotani_26_PTM}}] \label{thm:univ of PTM}
    Let $X$ be a conical symplectic variety with a projective crepant resolution $Y \to X$, and let $g_0$ be a hyperkähler cone metric on $X_{\mathrm{reg}}$.
\begin{enumerate}[itemsep=\medskipamount]
    \item (Construction of the PTM):\quad The hyperkähler cone metric $g_0$ determines a natural PTM $(\mathcal{Y}(1) \to \mathcal{C}(2), \omega_P, \tau_P)$, which is obtained by applying the $\mathbb{C}^*$-gluing construction to the universal Poisson deformation $\mathcal{Y}\to\mathcal{C}$ of $Y$.
    \item (Universality of the PTM):\quad Consider an algebraic hyperkähler structure $(Y, g, I, J, K)$ asymptotic to the cone metric $g_0$. Then, there exists a unique real section $s$ of the vector bundle $\mathcal{C}(2)$ such that the twistor space $Z$ corresponding to the metric $g$ is isomorphic, as a twistor model, to the pullback $Z_s$.
\end{enumerate}
\end{theorem}

\begin{remark}
    This theorem implies that the principal twistor model (PTM) completely determines the structure of the twistor spaces corresponding to an algebraic asymptotic hyperkähler structure, except for the family of twistor lines.
\end{remark}

Since the asymptotic condition strongly constrains the structure of these twistor lines, this universality yields the following injectivity proposition for the moduli space $\mathcal{M}$ of algebraic asymptotic hyperkähler structures on $Y$.

\begin{proposition}[{\cite[Cor.~4.12]{Kotani_26_PTM}}] \label{prop:PTM and period map}
    Under Setup~\ref{setup:general case}, the map $\Psi \colon \mathcal{M} \to H^0(\mathbb{P}^1, \mathcal{C}(2))^{\sigma_2}$, sending each class $m \in \mathcal{M}$ to its uniquely determined real section $s$, is injective.
\end{proposition}

\begin{remark}
    Via the isomorphism $\mathcal{C}\simeq H^2(Y;\mathbb{C})$ induced by the period map of the Poisson deformation (cf.~Def.~\ref{def: period map for PD}), the target space is identified as $H^0(\mathbb{P}^1,\mathcal{C}(2))^{\sigma_2}\simeq H^2(Y;\mathbb{R})\otimes\mathbb{R}^3$.
\end{remark}

%% file: source/Period_Map_and_Period_Domain.tex
\section{Period Map and Period Domain}\label{sec:period map and period domain}

In this section, under Setup~\ref{setup:general case}, we define the period map and the period domain for the moduli space of algebraic asymptotic hyperkähler structures. We show that the injectivity of the period map follows from the universality of the principal twistor model (PTM), and verify that its image is contained in the period domain defined via the wall of the universal Poisson deformation.

\subsection{Definition and Injectivity of the Period Map}

First, we define the period map from the moduli space to the space of triples of cohomology classes, and show that its injectivity immediately follows from the universality of the PTM \cite{Kotani_26_PTM}.

\medskip
The period map is defined as follows.
\begin{definition}\label{def:period map}
 Assume Setup~\ref{setup:general case}. For each element $m\in \mathcal{M}$, let $(Y,g,I,J,K)$ be a representative hyperkähler structure for $m$. Let $\omega_A$ be the Kähler form associated with the complex structure $A\in \{I,J,K\}$ of the hyperkähler metric $g$. Let $[\omega_A]\in H^2(Y;\mathbb{R})$ denote the period (i.e., the cohomology class) of $\omega_A$. Then, the period map $p$ is defined by:
 \begin{align*}
   p \colon \mathcal{M} &\to H^2(Y;\mathbb{R})\otimes\mathbb{R}^3\\
   m & \mapsto ([\omega_I],[\omega_J],[\omega_K])
 \end{align*}
\end{definition}

We prove the following proposition.

\begin{proposition}[{\cite[Theorem~3.43, Cor.~4.12]{Kotani_26_PTM}}]\label{prop:inj of period map}
 The period map $p \colon \mathcal{M}\to H^2(Y;\mathbb{R})\otimes\mathbb{R}^3$ is well-defined and injective.
\end{proposition}

\begin{proof}
 Let $m\in \mathcal{M}$, and let $\pi \colon Z\to\mathbb{P}^1$ be the twistor space corresponding to a representative hyperkähler structure $(Y,g,I,J,K)$ for $m$. Then, the relative holomorphic symplectic $2$-form $\Omega \in \Gamma(Z,\wedge^2{T'}^*_\pi\otimes\mathcal{O}(2))$ on $Z$ can be written as follows: for an affine coordinate $u\in\mathbb{C}_u \subset \mathbb{P}^1$,
 \[\Omega \coloneqq (\omega_J+i\omega_K)-2u\omega_I- u^2(\omega_J-i\omega_K).\]
 By taking the period of $\Omega$, we obtain a real section $s\in H^0(\mathbb{P}^1,\mathcal{C}(2))^{\sigma_2}$ of the vector bundle $\mathcal{C}(2) \coloneqq \mathcal{C}\otimes\mathcal{O}(2)$ defined by
 \[s \coloneqq [\omega_J+i\omega_K]-2u[\omega_I]- u^2[\omega_J-i\omega_K].\]
 Here, $\mathcal{C} \coloneqq H^2(Y;\mathbb{C})$.
 The correspondence $m\mapsto s$ obtained in this way is exactly the map $\Psi$ in Proposition~\ref{prop:PTM and period map}.
 Since the coefficients of the real section $s$ and the periods $([\omega_I],[\omega_J],[\omega_K])\in H^2(Y;\mathbb{R})\otimes\mathbb{R}^3$ are in one-to-one correspondence, the period map $p$ is well-defined and injective.
\end{proof}

\subsection{Period Domain}

Next, we introduce the period domain $\Omega$. Geometrically, $\Omega$ arises as a chamber separated by the wall of the universal Poisson deformation discussed in \S\ref{subsec:PD}. From the definition of this wall, it immediately follows that the image of the period map is contained in this domain.

\medskip
The period domain is defined as follows.
\begin{definition}\label{def:period domain}
 Under Setup~\ref{setup:general case}, let $D$ be the wall in $H^2(Y;\mathbb{C})$ as in Definition~\ref{def:wall}, and let its real part be given by the union of real hyperplanes $D_\mathbb{R} = \bigcup_{\alpha} H_{\alpha,\mathbb{R}} \subset H^2(Y;\mathbb{R})$.
 We define the corresponding wall in $H^2(Y;\mathbb{R})\otimes\mathbb{R}^3$ as
 \[
     D_\mathbb{R}\otimes\mathbb{R}^3 \coloneqq \bigcup_{\alpha} H_{\alpha,\mathbb{R}} \otimes \mathbb{R}^3.
 \]
 Then, we define the period domain $\Omega$ as the complement of this wall~(i.e.,~the chamber):
 \[
     \Omega \coloneqq (D_\mathbb{R}\otimes\mathbb{R}^3)^c \subset H^2(Y;\mathbb{R})\otimes\mathbb{R}^3.
 \]
\end{definition}

We prove the following proposition.
\begin{proposition}\label{prop:inclusion to period domain}
 Consider the setup of Definition~\ref{def:period map}, and let $p \colon \mathcal{M}\to H^2(Y;\mathbb{R})\otimes\mathbb{R}^3$ be the period map.
 Then, the image $p(\mathcal{M})$ is contained in the period domain $\Omega$.
\end{proposition}

\begin{proof}
 Suppose for a contradiction that there exists a hyperkähler structure $(Y,g,I,J,K) \in \mathcal{M}$ whose periods $([\omega_I],[\omega_J],[\omega_K])$ are contained in the wall $D_\mathbb{R}\otimes\mathbb{R}^3$.
 Then, by the definition of the wall $D$ (Def.~\ref{def:wall}), there exists a compact curve $\Sigma_\alpha$ on $Y_0\subset Y$ such that for each $A\in\{I,J,K\}$,
 \[\langle\Sigma_{\alpha},\omega_A\rangle \coloneqq \int_{\Sigma_{\alpha}}\omega_A=0.\]
 In particular, for the $I$-holomorphic symplectic $2$-form $\omega_\mathbb{C} \coloneqq \omega_J+i\omega_K$, we have $\langle\Sigma_{\alpha},\omega_\mathbb{C}\rangle=0$. As in the proof of Proposition~\ref{prop:meaning of wall}, this implies that $\Sigma_{\alpha}$ is an $I$-holomorphic curve on $Y$.
 However, the volume of $\Sigma_\alpha$ with respect to the Kähler form $\omega_I$ is zero, since $\langle\Sigma_{\alpha},\omega_I\rangle=0$.
 This implies that $g$ degenerates on $\Sigma_{\alpha}$, which contradicts the non-degeneracy of $g$.
 Therefore, $p(\mathcal{M})\subset \Omega$.
\end{proof}

%% file: source/HK_Quot_and_Surj_of_Period_Map.tex
\section{Hyperkähler Quotients and Surjectivity of the Period Map}\label{sec:HK quot and surj of the period map}
  In this section, we first provide the standard setup for hyperkähler quotients and describe them algebraically via GIT quotients. This algebraic description allows us to clarify their connection to Poisson deformations. After establishing the necessary asymptotic properties of the metrics, we prove a Torelli-type theorem for the moduli space of algebraic asymptotic hyperkähler structures on hyperkähler quotients.

\subsection{Setup for Hyperkähler Quotients}\label{subsec:setup_HK}

  The setup for hyperkähler moment maps treated in this and subsequent sections is as follows:
\begin{setup}\label{setup:HK quot}
  We naturally regard the quaternionic vector space $\mathbb{H}^n$ as a hyperkähler manifold $M=(\mathbb{H}^n, g_{\mathrm{std}}, I, J, K)$. Let $G\subset Sp(n)$ be a compact Lie group. We define the standard hyperkähler moment map $\mu \colon M \to\mathfrak{g}^*\otimes\mathbb{R}^3$ for $G$ by
  \[\langle\mu(x), A \rangle = \frac{1}{2}x^*Ax \in \mathrm{Im}(\mathbb{H})\simeq \mathbb{R}^3,\]
  where $x\in\mathbb{H}^n=M$ and $A\in\mathfrak{g}\subset \mathfrak{sp}(n)$.
  Also, let $\mathfrak{z}^* \coloneqq (\mathfrak{g}^*)^G$, and for each $\zeta\in \mathfrak{z}^*\otimes\mathbb{R}^3$, we define the hyperkähler quotient by $X_\zeta \coloneqq \mu^{-1}(\zeta)/G$.
\end{setup}

\begin{remark}
  If the Lie group $G$ acts freely on the non-empty fiber $\mu^{-1}(\zeta)$, then the real dimension of the hyperkähler quotient $X_\zeta$ is $4(n - \dim G)$.
\end{remark}

\begin{remark}
  Note that under the natural $\mathbb{R}^+$-action on $M$, the standard hyperkähler moment map $\mu$ satisfies $r^*\mu = r^2 \mu$ for any $r > 0$.
\end{remark}

\begin{remark}\label{rem:natural metric on HK quot}
  Let $g_\zeta$ be the natural hyperkähler metric induced on the hyperkähler quotient $X_\zeta$. Let $i \colon \mu^{-1}(\zeta) \to M$ be the inclusion map and $q \colon \mu^{-1}(\zeta) \to X_\zeta$ the natural projection. Then, the following relation holds \cite{HKLR}:
  \begin{equation}
    i^*g_{\mathrm{std}} = q^*g_\zeta. \label{eq:HK quot property}
  \end{equation}
\end{remark}

In this section, we impose three standing assumptions.

\begin{assumption}\label{assump:standing_assumptions}
  Throughout Section \ref{sec:HK quot and surj of the period map}, we assume the following three conditions:
  \begin{enumerate}
    \item The stable locus for the central fiber of the complex moment map is non-empty, i.e., $\mu_\mathbb{C}^{-1}(0)^s \neq \emptyset$.
    \item There exists a generic lattice point $\beta \in \mathfrak{z}^*_\mathbb{Z}$ such that the group $G_\mathbb{C}$ acts freely on the stable locus $M^{\beta-s}$ (cf.~\S \ref{subsec:GIT quot}).
    \item For the generic lattice point $\beta$ above, the Kirwan map $\kappa_{\beta}$ for the hyperkähler quotient $Y_{\beta} \coloneqq \mu^{-1}(\beta,0,0)/G \simeq \mu_\mathbb{C}^{-1}(0)/\!/_{\beta} G_\mathbb{C}$ is surjective (cf.~\S \ref{subsec:PD of HK quot}).
  \end{enumerate}
\end{assumption}

\begin{remark}\label{rem:roles_of_assumptions}
  The roles of these assumptions are summarized as follows:
  Assumptions (1) and (2) guarantee that the hyperkähler quotient $Y_\beta$ is a projective crepant resolution of the central quotient $X_0$.
  Assumptions (2) and (3) guarantee the existence of a natural period-surjective Poisson deformation $\mathcal{Y}^\beta$ of $Y_\beta$.
  Together, these assumptions are essential for proving the surjectivity of the period map in the main theorem.
\end{remark}

\subsection{Algebraic Description via GIT quotients}\label{subsec:GIT quot}\quad

In this subsection, we prepare an algebraic description of hyperkähler quotients via GIT quotients using the Kempf--Ness type theorem for affine varieties \cite{King94}. Based on this description, we show that generic hyperkähler quotients $Y_{\beta}$ naturally provide projective crepant resolutions of the central quotient $X_0$. This algebraic perspective plays a key role in connecting hyperkähler quotients to Poisson deformations in the subsequent subsection.

\subsubsection{Stability and GIT Quotients}
We briefly recall the notion of stability and GIT quotients introduced by King \cite{King94}.
First, we define the character lattice $\mathfrak{z}^*_\mathbb{Z}$ on $\mathfrak{z}^*$.

\begin{definition}\label{def:character lattice}
  Consider Setup~\ref{setup:HK quot}. Let $G_\mathbb{C}$ be the complexification of the compact Lie group $G$. Consider the character group $\mathrm{Hom}(G_\mathbb{C},\mathbb{C}^*)$ of the Lie group $G_\mathbb{C}$.
  Then, there is a natural inclusion map $\mathrm{Hom}(G_\mathbb{C},\mathbb{C}^*) \hookrightarrow \mathfrak{z}^*$ associating to each character $\chi \colon G_\mathbb{C}\to\mathbb{C}^*$ the element $\beta=(2\pi i)^{-1}d\chi \in\mathfrak{z}^*$.
  The image of this map is denoted by $\mathfrak{z}^*_\mathbb{Z} \subset\mathfrak{z}^*$, and we call it the character lattice of the Lie group $G$.
\end{definition}

We recall the notion of GIT quotients as follows.

\begin{definition}[{cf.~\cite[Def.~2.1]{King94}}]\label{def:b-ss}
  Consider Setup~\ref{setup:HK quot}. Let $V\subset M$ be a $G_\mathbb{C}$-invariant $I$-holomorphic subvariety. Let $R=\mathbb{C}[V]$ be its coordinate ring, and let $\chi \colon G_\mathbb{C}\to\mathbb{C}^*$ be the character corresponding to a lattice point $\beta \in \mathfrak{z}^*_\mathbb{Z}$ (i.e., $d\chi=2\pi i \beta$).
  For each $n\in\mathbb{Z}_{\geq 0}$, we define a subspace $R_{\chi^n}^{G_\mathbb{C}}$ of $R$ as follows:
  \[R_{\chi^n}^{G_\mathbb{C}} \coloneqq \{f\in R\mid f(g\cdot x)=\chi(g)^nf(x)\ (\text{for all }g\in G_\mathbb{C}, x\in V)\}.\]

  We define the stability for each point in $V$ as follows.
  \begin{enumerate}[itemsep=\medskipamount]
    \item A point $x \in V$ is \textit{$\beta$-semistable} if there exist an integer $n \geq 1$ and a function $f\in R_{\chi^n}^{G_\mathbb{C}}$ such that $f(x)\neq 0$.
    \item A point $x \in V$ is \textit{$\beta$-stable} if there exist an integer $n \geq 1$ and a function $f\in R_{\chi^n}^{G_\mathbb{C}}$ satisfying the following:
    \begin{enumerate}[label=(\roman*)]
      \item $f(x) \neq 0$.
      \item The action of $G_\mathbb{C}$ on the affine open set $V_f \coloneqq \{ y \in V \mid f(y) \neq 0 \}$ is closed (i.e., all $G_\mathbb{C}$-orbits in $V_f$ are closed sets).
      \item The dimension of the isotropy subgroup of $G_\mathbb{C}$ at $x$ is $0$ (i.e., $\dim (G_\mathbb{C})_x = 0$).
    \end{enumerate}
  \end{enumerate}
  We denote the $\beta$-semistable subset and the $\beta$-stable subset of $V$ by $V^{\beta-ss}$ and $V^{\beta-s}$, respectively.
  Also, we define the $\beta$-GIT quotient $V/\!/_{\beta}G_\mathbb{C}$ as follows:
  \[V/\!/_{\beta}G_\mathbb{C} \coloneqq \operatorname{Proj} \left(\bigoplus_{n=0}^\infty R_{\chi^n}^{G_\mathbb{C}}\right).\]
  Here, the right-hand side represents the projective scheme of the graded ring.
\end{definition}

\begin{remark}
  When $\beta=0$, we denote the semistable (resp.~stable) set by $V^{ss}$ (resp.~$V^s$). In this case, the GIT quotient coincides with the affine quotient $V/\!/G_\mathbb{C}$.
\end{remark}

\begin{remark}\label{rem:geometric meaning of GIT quot}
  The geometric properties of the $\beta$-GIT quotient $V/\!/_{\beta}G_\mathbb{C}$ are as follows (cf.~{\cite[p.~518]{King94}}).
  \begin{enumerate}[itemsep=\medskipamount]
    \item The $\beta$-GIT quotient $V/\!/_{\beta}G_\mathbb{C}$ constructed in Definition~\ref{def:b-ss} geometrically gives a good quotient of the semistable set $V^{\beta-ss}$. That is, there exists a natural surjective map $\varphi \colon V^{\beta-ss} \to V/\!/_{\beta}G_\mathbb{C}$, and the points of $V/\!/_{\beta}G_\mathbb{C}$ are in one-to-one correspondence with the GIT equivalence classes, which identify points where the closures of the orbits in $V^{\beta-ss}$ intersect each other.

    \item Furthermore, by the restriction of $\varphi$, the geometric quotient $V^{\beta-s}/G_\mathbb{C}$ of the stable set is naturally embedded as an open set of $V/\!/_{\beta}G_\mathbb{C}$.
  \end{enumerate}
\end{remark}

To determine stability, we recall the Hilbert--Mumford criterion:

\begin{proposition}[{\cite[Prop.~2.5]{King94}}]\label{prop:Hilbert-Munford criterion}
  Under Setup~\ref{setup:HK quot}, let $V\subset M$ be a $G_\mathbb{C}$-invariant $I$-holomorphic subvariety.
  Fix a lattice point $\beta \in \mathfrak{z}^*_\mathbb{Z}$, and let $\chi \colon G_\mathbb{C} \to\mathbb{C}^*$ be the corresponding character (i.e., $d\chi=2\pi i\beta$).
  For any point $x\in V$, we have the following characterizations:
  \begin{enumerate}
    \item $x\in V^{\beta-ss}$ if and only if for any $1$-parameter subgroup $\lambda \colon \mathbb{C}^*\to G_\mathbb{C}$ such that the limit \mbox{$\displaystyle\lim_{t\to 0} \lambda(t)\cdot x$} exists in $V$, we have $\langle\chi, \lambda\rangle \geq 0$.
    \item $x\in V^{\beta-s}$ if and only if $x\in V^{\beta-ss}$ and there exists no non-trivial $1$-parameter subgroup $\lambda \colon \mathbb{C}^*\to G_\mathbb{C}$ such that the limit \mbox{$\displaystyle\lim_{t\to 0} \lambda(t)\cdot x$} exists in $V$ and $\langle\chi, \lambda\rangle=0$.
  \end{enumerate}
  Here, $\langle\chi, \lambda\rangle$ denotes the natural pairing, given by the weight of the composition $\chi\circ\lambda$.
\end{proposition}

\begin{remark}\label{rem:to stable locus}
  As an immediate consequence of Proposition~\ref{prop:Hilbert-Munford criterion}, any stable point $x\in V^s$ is trivially $\beta$-stable for any choice of $\beta \in \mathfrak{z}^*_\mathbb{Z}$. Therefore, we always have the inclusion $V^s\subset V^{\beta-s}$.
\end{remark}

Proposition~\ref{prop:Hilbert-Munford criterion} implies the following lemma.

\begin{lemma}\label{lemma:generic_beta and stability}
  Under Setup~\ref{setup:HK quot}, for a generic lattice point $\beta\in\mathfrak{z}^*_\mathbb{Z}$, we have $M^{\beta-ss}=M^{\beta-s}$.
\end{lemma}

\begin{proof}
Fix a maximal torus $T_\mathbb{C}$ of $G_\mathbb{C}$, and let $X_*(T) \coloneqq \mathrm{Hom}(\mathbb{C}^*, T_\mathbb{C})$ be its cocharacter group, i.e., the group of $1$-parameter subgroups of $T_\mathbb{C}$.
  Any $1$-parameter subgroup $\lambda'$ of $G_\mathbb{C}$ is conjugate to some $\lambda \in X_*(T)$, meaning $\lambda'(t) = g\lambda(t)g^{-1}$ for some $g \in G_\mathbb{C}$.
  Since $\chi$ is conjugation-invariant, we have $\langle\chi, \lambda'\rangle = \langle\chi, \lambda\rangle$.
  Furthermore, the limit $\displaystyle\lim_{t\to 0} \lambda'(t)\cdot x$ exists if and only if $\displaystyle\lim_{t\to 0} \lambda(t)\cdot (g^{-1}x)$ exists.
  Therefore, to apply Proposition~\ref{prop:Hilbert-Munford criterion}, it suffices to restrict our attention to $\lambda \in X_*(T)$. \medskip

  Consider the weight decomposition $M=\bigoplus_{\mu \in S} M_\mu$ with respect to the $T_\mathbb{C}$-action, where $X^*(T) \coloneqq \mathrm{Hom}(T_\mathbb{C}, \mathbb{C}^*)$ is the character group of $T_\mathbb{C}$, $M_\mu \coloneqq \{x\in M \mid t\cdot x = \mu(t) x \text{ for all } t\in T_\mathbb{C}\}$, and $S \subset X^*(T)$ is the finite set of weights.
  For any point $x=\sum_{\mu \in S}x_\mu \in M$, an element $\lambda \in X_*(T)$ acts by
  \[ \lambda(t)\cdot x = \sum_{\mu \in S}t^{\langle\mu,\lambda\rangle}x_\mu. \]
  Hence, the limit $\displaystyle\lim_{t\to 0} \lambda(t)\cdot x$ exists if and only if $\langle\mu,\lambda\rangle \geq 0$ for all $\mu \in S$ with $x_\mu \neq 0$. \medskip

  In the vector space $\mathfrak{t}_\mathbb{R} \coloneqq X_*(T)\otimes_\mathbb{Z}\mathbb{R}$, we define a convex cone
  \begin{equation*}
    C_x \coloneqq \bigcap_{\mu \in S, \, x_\mu \neq 0}\{\lambda \in \mathfrak{t}_\mathbb{R} \mid \langle\mu,\lambda\rangle\geq 0\}.
  \end{equation*}
  Since the set of weights $S$ is finite, there are only finitely many such cones as $x$ varies over $M$. By the Minkowski--Weyl theorem, each cone is finitely generated, allowing us to choose a finite subset $\{\lambda_k\}_{k=1}^N \subset X_*(T)$ containing the generators for all these cones.\medskip

  For any lattice point $\beta \in\mathfrak{z}^*_\mathbb{Z}$ with corresponding character $\chi$, the condition $\langle \chi, \lambda_k \rangle = 0$ defines a hyperplane $H_k \subset \mathfrak{z}^*$. Let $\mathcal{W} \coloneqq \bigcup_{k=1}^N H_k$ be the finite union of these hyperplanes. Then, any choice of $\beta \in \mathfrak{z}^*_\mathbb{Z} \setminus \mathcal{W}$ ensures that $\langle\chi, \lambda\rangle \neq 0$ for all non-trivial $\lambda$ for which the limit exists. Consequently, Proposition~\ref{prop:Hilbert-Munford criterion} yields $M^{\beta-ss}=M^{\beta-s}$.
\end{proof}

\begin{remark}\label{rem:smoothness of generic beta-GIT quot}
  Let $V\subset M$ be a $G_\mathbb{C}$-invariant $I$-holomorphic subvariety.
  By Lemma~\ref{lemma:generic_beta and stability}, for a generic lattice point $\beta\in\mathfrak{z}^*_\mathbb{Z}$, we have $V^{\beta-ss}=V^{\beta-s}$.
  Consequently, if $V^{\beta-s}\neq \emptyset$, then the $\beta$-GIT quotient $V/\!/_\beta G_\mathbb{C}$ is generally an orbifold (i.e., it has at most finite quotient singularities).
  Furthermore, if the group $G_\mathbb{C}$ acts freely on $V^{\beta-s}$, then the quotient $V/\!/_\beta G_\mathbb{C}$ becomes a non-singular algebraic variety.
\end{remark}

Based on the hyperplanes constructed in the proof of Lemma~\ref{lemma:generic_beta and stability}, we formalize the stability walls in the parameter space $\mathfrak{z}^*$.

\begin{definition}\label{def:stability wall}
  Under the notation in the proof of Lemma~\ref{lemma:generic_beta and stability}, we define the \textit{stability wall} for $M$ as the finite union of hyperplanes $\mathcal{W} \coloneqq \bigcup_{k=1}^N H_k \subset \mathfrak{z}^*$.
  For any $G_\mathbb{C}$-invariant $I$-holomorphic subvariety $V \subset M$, we define the \textit{stability wall} for $V$, denoted by $\mathcal{W}(V) \subset \mathcal{W}$, as the finite union of the hyperplanes associated with the $1$-parameter subgroups that admit limits in $V$.
\end{definition}

Within each \textit{stability chamber} (i.e.,~a connected component of $\mathfrak{z}^*\setminus \mathcal{W}(V)$), the stability condition on $V$ remains completely invariant, as shown in the following lemma.

\begin{lemma}\label{lemma:same_chamber_same_stability}
  Consider Setup~\ref{setup:HK quot}. Let $V \subset M$ be a $G_\mathbb{C}$-invariant $I$-holomorphic subvariety, and let $\mathcal{W}(V) \subset \mathfrak{z}^*$ be the stability wall for $V$.
  If two generic lattice points $\beta, \beta' \in \mathfrak{z}^*_\mathbb{Z} \setminus \mathcal{W}(V)$ belong to the same stability chamber, then their semistable loci on $V$ coincide, i.e.,
  \[ V^{\beta-ss} = V^{\beta'-ss}. \]
\end{lemma}

\begin{proof}
  Let $\chi$ and $\chi'$ be the characters corresponding to $\beta$ and $\beta'$.
  As shown in the proof of Lemma~\ref{lemma:generic_beta and stability}, the limits in $V$ are governed by a finite set of $1$-parameter subgroups $\{\lambda_k\}$, which define the hyperplanes in $\mathcal{W}(V)$.
  For each $k$, the pairing $\langle -, \lambda_k \rangle$ is a linear functional on $\mathfrak{z}^*$.
  Since $\beta$ and $\beta'$ lie in the same connected component of $\mathfrak{z}^* \setminus \mathcal{W}(V)$, these continuous functionals cannot cross zero or change sign between them.
  Thus, $\langle \chi, \lambda_k \rangle$ and $\langle \chi', \lambda_k \rangle$ share the exact same strict sign for all $k$.
  By Proposition~\ref{prop:Hilbert-Munford criterion}~(1), $x \in V$ is $\beta$-semistable if and only if $\langle \chi, \lambda_k \rangle \geq 0$ for the specific generators $\lambda_k$ admitting limits for $x$.
  Thus, the semistability condition is completely equivalent for $\beta$ and $\beta'$.
  Therefore, $V^{\beta-ss} = V^{\beta'-ss}$.
\end{proof}

\subsubsection{Kempf--Ness Type Theorem and Projective Crepant Resolutions}\quad

Now we apply the GIT framework to our hyperkähler setup. The following theorem by King \cite{King94} connects the differential geometric quotient $X_\zeta$ with the algebraic GIT quotient. Using this isomorphism, together with Assumption \ref{assump:standing_assumptions}~(1) and (2), we show that generic hyperkähler quotients naturally provide projective crepant resolutions.

\begin{theorem}[{King \cite[Cor.~6.2]{King94}}]\label{thm:King}
  Consider Setup~\ref{setup:HK quot}. Let $\mu=(\mu_I,\mu_\mathbb{C}) \colon M\to\mathfrak{g}^*\oplus\mathfrak{g}^*_\mathbb{C}$. Fix an arbitrary lattice point $\beta \in \mathfrak{z}^*_\mathbb{Z}$ and a point $\alpha \in\mathfrak{z}^*_\mathbb{C}$, and let $\zeta=(\beta,\alpha) \in\mathfrak{z}^*\oplus\mathfrak{z}^*_\mathbb{C}$.
  Then, the inclusion map $\mu^{-1}(\zeta) \hookrightarrow \mu_\mathbb{C}^{-1}(\alpha)^{\beta-ss}$ induces a Kempf--Ness type homeomorphism:
  \[X_\zeta \simeq \mu_\mathbb{C}^{-1}(\alpha)/\!/_{\beta}G_\mathbb{C}.\]
\end{theorem}

\begin{proof}
  By \cite[Cor.~6.2]{King94}, the inclusion $\mu_I^{-1}(\beta) \hookrightarrow M^{\beta-ss}$ naturally induces a Kempf--Ness type isomorphism with respect to the moment map $\mu_I$:
  \[
    \mu_I^{-1}(\beta)/G \simeq M/\!/_{\beta} G_\mathbb{C}.
  \]
  By restricting this inclusion to the affine subvariety $V \coloneqq \mu_\mathbb{C}^{-1}(\alpha) \subset M$, we obtain
  \[
    \mu_I^{-1}(\beta)\cap V \hookrightarrow M^{\beta-ss} \cap V = V^{\beta-ss},
  \]
  which immediately yields the desired isomorphism $X_\zeta \simeq V/\!/_{\beta}G_\mathbb{C}$.
\end{proof}

\begin{remark}\label{rem:to King}
  The complex structure $I$ and the $I$-holomorphic symplectic form on $M$ naturally induce those on both sides of this isomorphism. Therefore, the Kempf--Ness type isomorphism provides an isomorphism as $I$-holomorphic symplectic varieties.
\end{remark}

As a consequence of Assumption~\ref{assump:standing_assumptions}~(1), the existence of stable points naturally extends to any $\alpha \in \mathfrak{z}^*_\mathbb{C}$.

\begin{lemma}\label{lemma:V_alpha is not empty}
  Consider Setup~\ref{setup:HK quot} and let $\mu=(\mu_I,\mu_\mathbb{C}) \colon M\to\mathfrak{g}^*\oplus\mathfrak{g}^*_\mathbb{C}$. If $\mu_\mathbb{C}^{-1}(0)^s \neq \emptyset$~(Assumption~\ref{assump:standing_assumptions}~(1)), then for any $\alpha\in\mathfrak{z}^*_\mathbb{C}$, $\mu_\mathbb{C}^{-1}(\alpha)^s \neq \emptyset$.
\end{lemma}

\begin{proof}
  For a point $\alpha\in\mathfrak{z}^*_\mathbb{C}$, the condition $\mu_\mathbb{C}^{-1}(\alpha)^s \neq \emptyset$ is an open condition. Hence, when $\mu_\mathbb{C}^{-1}(0)^s \neq \emptyset$, there exists an open neighborhood of the point $0\in\mathfrak{z}^*_\mathbb{C}$ such that for any point $\alpha'$ belonging to it, $\mu_\mathbb{C}^{-1}(\alpha')^s \neq \emptyset$. Now, by Setup~\ref{setup:HK quot}, the complex moment map $\mu_\mathbb{C}$ satisfies the compatibility condition $r^*\mu_\mathbb{C} = r^2 \mu_\mathbb{C}$ with respect to the natural $\mathbb{R}^+$-action on $M$. Therefore, it follows that $\mu_\mathbb{C}^{-1}(\alpha)^s \neq \emptyset$ for any $\alpha$.
\end{proof}

Combined with Lemma~\ref{lemma:generic_beta and stability} and the free action assumption, this guarantees the non-singularity of the GIT quotients for generic lattice points.

\begin{proposition}\label{prop:smoothness of generic beta-GIT quot}
  Under Setup~\ref{setup:HK quot} and Assumptions~\ref{assump:standing_assumptions}~(1) and (2), let $\mu=(\mu_I,\mu_\mathbb{C}) \colon M\to\mathfrak{g}^*\oplus\mathfrak{g}^*_\mathbb{C}$. Then for a generic lattice point $\beta\in\mathfrak{z}^*_\mathbb{Z}$ and any point $\alpha\in\mathfrak{z}^*_\mathbb{C}$, the $\beta$-GIT quotient $\mu_\mathbb{C}^{-1}(\alpha)/\!/_\beta G_\mathbb{C}$ is non-singular.
\end{proposition}

\begin{proof}
  Let $V \coloneqq \mu_\mathbb{C}^{-1}(\alpha) \subset M$. Then, by Lemma~\ref{lemma:generic_beta and stability}, we have $V^{\beta-ss}=V^{\beta-s}$ for a generic $\beta$.
  By Lemma~\ref{lemma:V_alpha is not empty}, Assumption~\ref{assump:standing_assumptions}~(1) (i.e.,~$\mu_\mathbb{C}^{-1}(0)^s \neq \emptyset$) implies $V^s \neq \emptyset$. Furthermore, by Remark~\ref{rem:geometric meaning of GIT quot}, we obtain $\emptyset \neq V^s \subset V^{\beta-s}$.
  On the other hand, Assumption~\ref{assump:standing_assumptions}~(2) ensures that the group $G_\mathbb{C}$ acts freely on $V^{\beta-s} \subset M^{\beta-s}$.
  Consequently, by Remark~\ref{rem:smoothness of generic beta-GIT quot}, the $\beta$-GIT quotient $\mu_\mathbb{C}^{-1}(\alpha)/\!/_\beta G_\mathbb{C}$ is a non-singular algebraic variety.
\end{proof}

Next, we show that the natural map between hyperkähler quotients provides a projective crepant resolution whenever the domain is non-singular.

\begin{proposition}\label{prop:crepant resol and HK quot}
  Under Setup~\ref{setup:HK quot} and Assumption~\ref{assump:standing_assumptions}~(1),
  let $\mu=(\mu_I,\mu_\mathbb{C}) \colon M\to\mathfrak{g}^*\oplus\mathfrak{g}^*_\mathbb{C}$.
  Fix an arbitrary lattice point $\beta \in \mathfrak{z}^*_\mathbb{Z}$ and a point $\alpha \in\mathfrak{z}^*_\mathbb{C}$, and let $\zeta=(\beta,\alpha) \in\mathfrak{z}^*\oplus\mathfrak{z}^*_\mathbb{C}$ and $\zeta_0 \coloneqq (0,\alpha)$. Assume that $X_\zeta\simeq \mu^{-1}_\mathbb{C}(\alpha)/\!/_\beta G_\mathbb{C}$ is non-singular.
  Then, the map $\pi_I \colon X_\zeta \to X_{\zeta_0}$ induced by the inclusion map $\mu^{-1}(\zeta) \subset \mu^{-1}_\mathbb{C}(\alpha)=\mu^{-1}_\mathbb{C}(\alpha)^{ss}$ is an $I$-holomorphic projective crepant resolution.
\end{proposition}

\begin{proof}
  By Theorem~\ref{thm:King} and Remark~\ref{rem:to King}, biholomorphic equivalences $X_\zeta \simeq \mu_\mathbb{C}^{-1}(\alpha)/\!/_{\beta}G_\mathbb{C}$ and $X_{\zeta_0} \simeq \mu_\mathbb{C}^{-1}(\alpha)/\!/G_\mathbb{C}$ with respect to $I$ hold.
  Let $V \coloneqq \mu_\mathbb{C}^{-1}(\alpha)$. For Definition~\ref{def:b-ss}, consider the following projective morphism induced by the inclusion map $R^{G_\mathbb{C}}\subset \bigoplus_{n=0}^\infty R_{\chi^n}^{G_\mathbb{C}}$:
  \[\nu \colon V/\!/_{\beta}G_\mathbb{C} \to V/\!/G_\mathbb{C}.\]
  Note that this morphism $\nu$ is naturally identified with the map $\pi_I \colon X_\zeta \to X_{\zeta_0}$ induced by the inclusion map $\mu^{-1}(\zeta) \subset \mu^{-1}_\mathbb{C}(\alpha)$ stated in the lemma.
  Since the stable locus satisfies $\mu_\mathbb{C}^{-1}(0)^s \neq \emptyset$, by Lemma~\ref{lemma:V_alpha is not empty} and Remark~\ref{rem:geometric meaning of GIT quot}, noting that $\emptyset \neq V^s\subset V^{\beta-s}$, the geometric quotient $V^s/G_\mathbb{C}$ is embedded as a dense open subset of $V/\!/_{\beta}G_\mathbb{C}$ and $V/\!/G_\mathbb{C}$, respectively.
  Then, since the morphism $\nu$ is bijective on $V^s/G_\mathbb{C}$, $\nu$ is a birational map.
  Furthermore, since the morphism $\nu$ preserves the natural holomorphic symplectic $2$-form, if $X_\zeta$ is non-singular, then the morphism $\nu$ is a projective crepant resolution. From this, the assertion of the lemma follows.
\end{proof}

\begin{remark}\label{rem:for crepant resol of HK quot}
  \begin{enumerate}
    \item When the target space $X_{\zeta_0}$ is non-singular, the map $\pi_I$ provides an $I$-biholomorphic equivalence $X_{\zeta}\simeq X_{\zeta_0}$.
    \item If the $J$- and $K$-components of the point $\zeta$ are lattice points, then one can similarly obtain a $J$-holomorphic map $\pi_J$ and a $K$-holomorphic map $\pi_K$ that provide projective crepant resolutions via hyperkähler rotation.
  \end{enumerate}
\end{remark}

By combining Proposition~\ref{prop:smoothness of generic beta-GIT quot} and Proposition~\ref{prop:crepant resol and HK quot}, we obtain the following corollary, which explicitly achieves the goal of this subsection.

\begin{corollary}\label{cor:crepant_resol_Y_beta}
  Under Setup~\ref{setup:HK quot} and Assumption~\ref{assump:standing_assumptions}~(1) and (2), for a generic lattice point $\beta \in \mathfrak{z}^*_\mathbb{Z}$, the hyperkähler quotient $Y_\beta \coloneqq X_{(\beta,0,0)}$ is a projective crepant resolution of the central quotient $X_0$.
\end{corollary}

\begin{proof}
  By Proposition~\ref{prop:smoothness of generic beta-GIT quot}, $Y_\beta$ is non-singular for a generic $\beta \in\mathfrak{z}^*_\mathbb{Z}$. Applying Proposition~\ref{prop:crepant resol and HK quot} with $\alpha=0$, the natural map $\pi_I \colon Y_\beta \to X_0$ gives a projective crepant resolution.
\end{proof}

\subsection{Poisson Deformations of Hyperkähler Quotients}\label{subsec:PD of HK quot}\quad

In this subsection, we describe the relationship between hyperkähler quotients and Poisson deformations, and clarify the properties of their period maps.

\subsubsection{The Kirwan Map and the Period Map}

We first define the Kirwan map, which provides a topological description of the period map for Poisson deformations.
\begin{definition}\label{def:Kirwan map}
  Consider Setup~\ref{setup:HK quot}. For a generic lattice point $\beta \in\mathfrak{z}^*_\mathbb{Z}$, let $Y_\beta \coloneqq \mu_\mathbb{C}^{-1}(0)/\!/_\beta G_\mathbb{C}$. Assume that $G_\mathbb{C}$ acts freely on $\mu_\mathbb{C}^{-1}(0)^{\beta-ss}$ (and consequently $Y_\beta$ is non-singular).  Then, we define the morphism $\kappa_{\beta,\mathbb{Z}}$ induced by the inclusion map $\mu_\mathbb{C}^{-1}(0)^{\beta-ss} \subset M$ as follows:
  \begin{equation}
    \kappa_{\beta,\mathbb{Z}} \colon \mathfrak{z}^*_\mathbb{Z}\simeq H^2_{G_\mathbb{C}}(M;\mathbb{Z})\to H^2_{G_\mathbb{C}}(\mu_\mathbb{C}^{-1}(0)^{\beta-ss};\mathbb{Z}) \simeq H^2(Y_\beta;\mathbb{Z}).
  \end{equation}
  We call this morphism $\kappa_{\beta,\mathbb{Z}}$ the (degree $2$) Kirwan map of $Y_\beta$.
\end{definition}

\begin{remark}\label{rem:Kirwan map}
  \begin{enumerate}
    \item The isomorphism of the target spaces follows from the fact that $G_\mathbb{C}$ acts freely on $\mu_\mathbb{C}^{-1}(0)^{\beta-ss}$.
    \item Since $M=\mathbb{H}^n$ is contractible, $H^2_{G_\mathbb{C}}({M};\mathbb{Z})\simeq\mathfrak{z}^*_\mathbb{Z}$ holds.
    \item We will also refer to its complexification $\kappa_\beta \coloneqq \kappa_{\beta,\mathbb{Z}}\otimes_\mathbb{Z}\mathbb{C} \colon \mathfrak{z}^*_\mathbb{C}\to H^2(Y_\beta;\mathbb{C})$ simply as the Kirwan map.
  \end{enumerate}
\end{remark}

\begin{remark}\label{rem:meaning of Kirwan map}
  Geometrically, the Kirwan map $\kappa_{\beta,\mathbb{Z}}$ is interpreted as follows. For any lattice point $\beta' \in\mathfrak{z}^*_\mathbb{Z}$, let $\chi'$ be the corresponding character of $G_\mathbb{C}$. Let $L_{\chi'} \to Y_\beta$ be the line bundle associated with the principal $G_\mathbb{C}$-bundle $\mu_\mathbb{C}^{-1}(0)^{\beta-ss} \to Y_\beta$ via the character $\chi'$. Then, the Kirwan map is explicitly given by $\kappa_{\beta,\mathbb{Z}}(\beta') = c_1(L_{\chi'}) \in H^2(Y_\beta;\mathbb{Z})$.
\end{remark}

Under Assumption~\ref{assump:standing_assumptions}, the following proposition by Nagaoka \cite{Nagaoka21} guarantees that the complex moment map yields a Poisson deformation whose period map is surjective.
\begin{proposition}[{Nagaoka \cite{Nagaoka21}}]\label{prop:moment map and PD}
  Consider Setup~\ref{setup:HK quot}. Take a generic lattice point $\beta \in \mathfrak{z}^*_\mathbb{Z}$ satisfying Assumption~\ref{assump:standing_assumptions}.
  Then, the morphism
  \[\varphi_\beta \colon \mathcal{Y}^\beta \coloneqq \mu_\mathbb{C}^{-1}(\mathfrak{z}^*_\mathbb{C})/\!/_\beta G_\mathbb{C}\to \mathfrak{z}^*_\mathbb{C}\]
  induced by the complex moment map $\mu_\mathbb{C} \colon M \to\mathfrak{g}^*_\mathbb{C}$ defines a Poisson deformation of $Y_\beta \coloneqq \mu_\mathbb{C}^{-1}(0)/\!/_\beta G_\mathbb{C}$.
  Furthermore, its period map $\mathfrak{z}^*_\mathbb{C}\to H^2(Y_\beta;\mathbb{C})$ coincides with the Kirwan map $\kappa_\beta$, and is surjective by Assumption~\ref{assump:standing_assumptions}.
\end{proposition}

\begin{proof}
  We follow the argument by Nagaoka \cite{Nagaoka21}. We show the following two claims in order:
  \begin{enumerate}[wide=\parindent, itemsep=\medskipamount]
    \item $\mathcal{Y}^\beta$ is a Poisson deformation:\quad
    By Proposition~\ref{prop:crepant resol and HK quot}, $Y_\beta \coloneqq \mu_\mathbb{C}^{-1}(0)/\!/_\beta G_\mathbb{C}$ forms a projective crepant resolution of the central quotient $X_0 \coloneqq \mu_\mathbb{C}^{-1}(0)/\!/ G_\mathbb{C}$.
    By Proposition~\ref{prop:smoothness of generic beta-GIT quot}, each fiber of the morphism $\varphi_\beta$ is non-singular over $\mathfrak{z}^*_\mathbb{C}$.
    Since the dimension of each fiber is constant, by the criterion for flatness {\cite[Cor.~14.128]{Görtz-Wedhorn10}}, $\mathcal{Y}^\beta$ is a flat family.
    Moreover, since each fiber has a natural holomorphic symplectic form, it follows that $\mathcal{Y}^\beta$ is a Poisson deformation.

    \item The period map of $\mathcal{Y}^\beta$ coincides with the Kirwan map:\quad
    Using the algebro-geometric version of the Duistermaat--Heckman theorem, known in real symplectic geometry {\cite[Prop.~3.2.1]{Losev12}}, it follows that the period map of $\mathcal{Y}^\beta$ (cf.~Def.~\ref{def: period map for PD}) coincides with the Kirwan map $\kappa_\beta$.
  Now, by Assumption~\ref{assump:standing_assumptions}, the Kirwan map $\kappa_\beta$ is surjective, so the period map of the Poisson deformation $\mathcal{Y}^\beta$ is also surjective. Consequently, the assertion follows.
  \end{enumerate}
\end{proof}

\subsubsection{Discriminant Locus and Propagation of Period-surjectivity}\quad

Proposition \ref{prop:moment map and PD} provides a period-surjective Poisson deformation for a specific generic lattice point $\beta$. In order to characterize the non-singularity of hyperkähler quotients and propagate the surjectivity of the Kirwan map to other lattice points, we consider the affinization of the family $\mathcal{Y}^\beta$. The discriminant locus $D'$ of this affine family will serve as the foundation for determining the wall-chamber structure in the next subsection.

\begin{lemma}\label{lemma:to Xcal'}
  Similarly to Proposition~\ref{prop:moment map and PD}, let $\mathcal{X}'$ denote the morphism obtained when $\beta=0$:
  \[\mathcal{X}' \coloneqq \mu_\mathbb{C}^{-1}(\mathfrak{z}^*_\mathbb{C})/\!/G_\mathbb{C}\to \mathfrak{z}^*_\mathbb{C}.\]
  Then, $\mathcal{X}'$ defines an $I$-Poisson deformation of the central quotient $X_0$.
\end{lemma}

\begin{proof}
  The natural morphism $\nu_\beta \colon \mathcal{Y}^\beta \to\mathcal{X}'$ obtained by Proposition~\ref{prop:crepant resol and HK quot} is actually the affinization of the Poisson deformation $\mathcal{Y}^\beta$.
  Therefore, it follows that $\mathcal{X}'$ is an $I$-Poisson deformation of the central quotient $X_0$.
\end{proof}

\begin{remark}\label{rem:discriminant locus on z^*_C}
  In what follows, we denote the discriminant locus on $\mathfrak{z}^*_\mathbb{C}$ for the affine family $\mathcal{X}'$ by $D'$.
\end{remark}

Using this affine family $\mathcal{X}'$, we can characterize the smoothness of the hyperkähler quotient of type $X_{(\zeta_I,0,0)}$ via its discriminant locus.

\begin{lemma}\label{lemma:discriminant locus and HK moment map}
  Consider Setup~\ref{setup:HK quot} and Assumption~\ref{assump:standing_assumptions}.
  Let $D'\subset \mathfrak{z}^*_\mathbb{C}$ be the discriminant locus.
  Let its real part be $D'_\mathbb{R} \subset \mathfrak{z}^*$. Then, for any point $\zeta_I \in \mathfrak{z}^*$, we have $\zeta_I\notin D'_\mathbb{R}$ if and only if $X_{(\zeta_I,0,0)}$ is non-singular.
\end{lemma}

\begin{proof}
  Assume that $\zeta_I \notin D'_\mathbb{R}$.
  It suffices to show that $X_{(0,\zeta_I,0)}$, obtained by applying a hyperkähler rotation to $X_{(\zeta_I,0,0)}$, is non-singular.
  By Theorem~\ref{thm:King}, we have $X_{(0,\zeta_I,0)}\simeq \mathcal{X}'_{\zeta_I}$.
  Here, since $\zeta_I \notin D'$, $\mathcal{X}'_{\zeta_I}$ is non-singular.
  Therefore, we deduce that $X_{(\zeta_I,0,0)}$ is non-singular.
  The converse follows by the same argument, hence the assertion of the lemma holds.
\end{proof}

By relating the discriminant locus $D'_\mathbb{R}$ to the stability wall $\mathcal{W}(V)$ introduced in Section~\ref{subsec:GIT quot} (Def.~\ref{def:stability wall}), we can explicitly guarantee the invariance of the holomorphic symplectic structure within each chamber.

\begin{lemma}\label{lemma:same_chamber_PD}
  Consider Setup~\ref{setup:HK quot} and let $D'_\mathbb{R} \subset \mathfrak{z}^*$ be the real part of the discriminant locus $D'$.
  If two generic lattice points $\beta, \beta' \in \mathfrak{z}^*_\mathbb{Z} \setminus D'_\mathbb{R}$ belong to the same chamber (i.e.,~connected component) of $\mathfrak{z}^* \setminus D'_\mathbb{R}$, their corresponding Poisson deformations coincide:
  \[ \mathcal{Y}^\beta = \mathcal{Y}^{\beta'}. \]
\end{lemma}

\begin{proof}
  Let $V \coloneqq \mu_\mathbb{C}^{-1}(\mathfrak{z}^*_\mathbb{C})$. Recall from Section~\ref{subsec:GIT quot} that the stability condition on $V$ is controlled by the stability wall $\mathcal{W}(V) \subset \mathfrak{z}^*$, which is a finite union of hyperplanes.
  Suppose $\beta \notin D'_\mathbb{R}$. By Lemma~\ref{lemma:discriminant locus and HK moment map}, the Poisson deformation $\mathcal{Y}^\beta \simeq V/\!/_\beta G_\mathbb{C} \to \mathfrak{z}^*_\mathbb{C}$ is non-singular, which implies $V^{\beta-ss} = V^{\beta-s}$ i.e.,~$\beta \notin \mathcal{W}(V)$.
  Thus, we obtain $\mathcal{W}(V) \subset D'_\mathbb{R}$.\medskip

  We assume that $\beta$ and $\beta'$ belong to the same chamber of $\mathfrak{z}^* \setminus D'_\mathbb{R}$.
  Since they necessarily belong to the same stability chamber of $\mathfrak{z}^* \setminus \mathcal{W}(V)$, Lemma~\ref{lemma:same_chamber_same_stability} implies that their semistable loci coincide: $V^{\beta-ss} = V^{\beta'-ss}$.
  Thus, the Poisson deformations coincide:
  \[\mathcal{Y}^\beta=V/\!/_{\beta} G_\mathbb{C} = V/\!/_{\beta'}G_\mathbb{C}=\mathcal{Y}^{\beta'}.\]
\end{proof}

By avoiding this discriminant locus, the surjectivity of the Kirwan map propagates to other lattice points.

\begin{lemma}\label{lemma:all Kirwan surj}
  Suppose that Assumption~\ref{assump:standing_assumptions} holds. Then, for any lattice point $\beta \in \mathfrak{z}^*_\mathbb{Z} \setminus D'_\mathbb{R} \subset \mathfrak{z}^*$, $\kappa_\beta$ is surjective. That is, by Proposition~\ref{prop:moment map and PD}, a period-surjective Poisson deformation $\mathcal{Y}^\beta\to\mathfrak{z}^*_\mathbb{C}$ of $Y_\beta$ is obtained.
\end{lemma}

\begin{proof}
  By Lemma~\ref{lemma:discriminant locus and HK moment map}, $Y_{\beta}$ is non-singular.
  For any lattice point $\beta \in \mathfrak{z}^*_\mathbb{Z} \setminus D'_\mathbb{R} \subset \mathfrak{z}^*$, the Poisson deformation $\mathcal{Y}^\beta \to\mathfrak{z}^*_\mathbb{C}$ of $Y_\beta$ is constructed similarly to Proposition~\ref{prop:moment map and PD}.
  Now, since $\mathcal{Y}^\beta$ has a natural $\mathbb{C}^*$-action, the period map is linear.
  Combined with the fact that the natural morphism $\mathcal{Y}^\beta \to \mathcal{X}'$ by Proposition~\ref{prop:crepant resol and HK quot} determines the affinization, we conclude that the period map is surjective. Since the Kirwan map coincides with the period map, $\kappa_\beta$ is surjective in particular.
\end{proof}

Finally, by utilizing the natural $\mathbb{R}^+$-action on $M$, we scale the parameters to obtain a period-surjective Poisson deformation for any generic rational point.

\begin{proposition}\label{prop:PD for rational point}
  Consider Setup~\ref{setup:HK quot} and Assumption~\ref{assump:standing_assumptions}.
  Let $\mathfrak{z}^*_\mathbb{Q} \coloneqq \mathfrak{z}^*_\mathbb{Z}\otimes_\mathbb{Z}\mathbb{Q}$ denote the scalar extension to $\mathbb{Q}$ of the character lattice $\mathfrak{z}^*_\mathbb{Z}$.
  Then, for any rational point $\beta \in \mathfrak{z}^*_\mathbb{Q} \setminus D'_\mathbb{R}\subset\mathfrak{z}^*$, the family $\mathcal{Y}^\beta \coloneqq \{X_{(\beta,\alpha)}\}_{\alpha\in \mathfrak{z}^*_\mathbb{C}}\to\mathfrak{z}^*_\mathbb{C}$ is a period-surjective Poisson deformation of $Y_\beta \coloneqq (X_{(\beta,0,0)},I)$.
\end{proposition}

\begin{proof}
  For $\zeta\in\mathfrak{z}^*\otimes\mathbb{R}^3$, consider the scalar multiplication by $r>0$ acting as $X_{r^{-2}\zeta} \to X_{\zeta}$, which is induced by the $\mathbb{R}^+$-action on $M$ (cf.~Proposition~\ref{prop:preservation of asymp HK met}, Proof~(2)). Note that for the natural hyperkähler metric $g_\zeta$ on $X_\zeta$, the following equality holds:
  \begin{equation}
    r^*g_{\zeta} = r^2g_{r^{-2}\zeta}. \label{eq:R^+-action}
  \end{equation}
  Consider $Y_\beta \coloneqq (X_{(\beta,0,0)},I)$ and the family $\mathcal{Y}^\beta \coloneqq \{X_{(\beta,\alpha)}\}_{\alpha\in \mathfrak{z}^*_\mathbb{C}}\to\mathfrak{z}^*_\mathbb{C}$ for any rational point $\beta \in \mathfrak{z}^*_\mathbb{Q} \setminus D'_\mathbb{R} \subset \mathfrak{z}^*$.
  Let $\Omega^\beta_{\alpha}$ denote the natural $I$-holomorphic symplectic form on the fiber $\mathcal{Y}_\alpha^\beta$ of the family $\mathcal{Y}^\beta$.
  There exists a positive integer $N \in \mathbb{Z}_{>0}$ such that $N^2\beta \in \mathfrak{z}^*_\mathbb{Z}\setminus D'_\mathbb{R}$.
  Then, by Lemma~\ref{lemma:all Kirwan surj}, the family $\mathcal{Y}^{N^2\beta}$ gives a period-surjective Poisson deformation of $Y_{N^2\beta}$.
  Since the family $\mathcal{Y}^\beta$ and the family $\mathcal{Y}^{N^2\beta}$ are related by the scalar multiplication by $r=N>0$, they are isomorphic as deformations of complex manifolds.
  On the other hand, by relation \eqref{eq:R^+-action}, we have $r^*\Omega^{N^2\beta}_{N^2\alpha}=N^2\Omega^\beta_\alpha$, so the family $\mathcal{Y}^\beta$ is isomorphic to $(\mathcal{Y}^{N^2\beta}, N^{-2}\Omega^{N^2\beta})$ as a Poisson deformation. Here, $(\mathcal{Y}^{N^2\beta}, N^{-2}\Omega^{N^2\beta})$ is a period-surjective Poisson deformation of $Y^\beta \simeq (Y_{N^2\beta}, N^{-2}\Omega_0^{N^2\beta})$.
  Consequently, the assertion follows.
\end{proof}

\subsection{Asymptotic Hyperkähler Metrics on Hyperkähler Quotients}\label{subsec:AHK on HK quot}

In this subsection, we shift our focus to the analytic properties of the natural hyperkähler metric $g_\zeta$ on the general hyperkähler quotient $X_\zeta$. By fixing a projective crepant resolution $Y_{\beta_0}$ of $X_0$ as a reference space, our goal is to construct a diffeomorphism $\Phi_\zeta \colon X_\zeta\to Y_{\beta_0}$ for any non-singular $X_\zeta$, and show that the push-forward metric $(\Phi_\zeta)_*g_\zeta$ is an asymptotic hyperkähler metric on $Y_{\beta_0}$.

\subsubsection{Hyperkähler Cone and General Crepant Resolutions}

First, we recall that the central quotient $X_0$ is equipped with a natural hyperkähler cone metric $g_0$.
\begin{proposition}[{cf.~\cite[Prop.~4.26]{Kotani_26_PTM}}]\label{prop:cone met on central HK quot}
  Consider Setup~\ref{setup:HK quot}. Then, the central quotient $X_0$ is a conical symplectic variety, and the regular locus $(X_0)_{\mathrm{reg}}$ admits a natural hyperkähler cone metric $g_0$ induced by the hyperkähler moment map.
\end{proposition}

  \begin{proof}
  By Theorem~\ref{thm:King}, the central quotient $X_0$ is $I$-biholomorphic to the affine quotient $\mu_\mathbb{C}^{-1}(0)/\!/G_\mathbb{C}$, and thus it is an affine variety.
  By Setup~\ref{setup:HK quot}, for the natural $I$-holomorphic $\mathbb{C}^*$-action on $M$, $\lambda^*\mu_\mathbb{C} = \lambda^2\mu_\mathbb{C}$ holds.
  Moreover, since the $G_\mathbb{C}$-action on $M$ is linear, it commutes with the natural $\mathbb{C}^*$-action.
  Therefore, the central quotient $X_0\simeq\mu_\mathbb{C}^{-1}(0)/\!/G_\mathbb{C}$ admits a natural good $\mathbb{C}^*$-action.
  Since the standard metric $g_{\mathrm{std}}$ on $M$ naturally scales with weight $2$ under the $\mathbb{R}^+$-action, Remark~\ref{rem:natural metric on HK quot} directly implies that the induced hyperkähler metric $g_0$ on $(X_0)_{\mathrm{reg}}$ inherits this property. Thus, $g_0$ is a cone metric.
\end{proof}

The following result of Nakajima \cite{Nakajima94} provides an analytic generalization of Proposition~\ref{prop:crepant resol and HK quot}, extending the construction of crepant resolutions to general hyperkähler quotients.

\begin{proposition}[{Nakajima \cite[Theorem~4.1]{Nakajima94}}]\label{prop:crepant resol and HK quot (general case)}
  Consider Setup~\ref{setup:HK quot} and Assumption~\ref{assump:standing_assumptions}.
  Let $\mu=(\mu_I,\mu_\mathbb{C}) \colon M\to\mathfrak{g}^*\oplus\mathfrak{g}^*_\mathbb{C}$, and assume that the stable locus satisfies $\mu_\mathbb{C}^{-1}(0)^s \neq \emptyset$. Fix an arbitrary point $\zeta_I \in \mathfrak{z}^*$ and a point $\alpha \in\mathfrak{z}^*_\mathbb{C}$, and let $\zeta=(\zeta_I,\alpha) \in\mathfrak{z}^*\oplus\mathfrak{z}^*_\mathbb{C}$ and $\zeta_0 \coloneqq (0,\alpha)$. Assume that $X_\zeta$ is non-singular.
  Then, the map $\pi_I \colon X_\zeta \to X_{\zeta_0}$ induced by the inclusion map $\mu^{-1}(\zeta) \subset \mu^{-1}_\mathbb{C}(\alpha)=\mu^{-1}_\mathbb{C}(\alpha)^{ss}$ is an $I$-holomorphic (not necessarily projective) crepant resolution.
\end{proposition}

\begin{proof}
  Note that the hyperkähler quotient $X_{\zeta_0}$ is non-empty by Lemma~\ref{lemma:to Xcal'}.
  The properness and birationality of $\pi_I$ follow directly from the argument in \cite[Theorem~4.1]{Nakajima94}. Since $\pi_I$ preserves the holomorphic symplectic form, it gives a crepant resolution as in Proposition~\ref{prop:crepant resol and HK quot}.
\end{proof}

\begin{remark}\label{remark:to crepant resol prop}
  \begin{enumerate}
    \item When $X_{\zeta_0}$ is non-singular, an $I$-biholomorphic equivalence $X_\zeta \simeq X_{\zeta_0}$ is obtained.
  \item By hyperkähler rotation, we similarly obtain $J$- and $K$-holomorphic crepant resolutions $\pi_J$ and $\pi_K$.
    \item When $\zeta_I$ is a rational point, that is, $\zeta_I\in \mathfrak{z}^*_\mathbb{Q} \coloneqq \mathfrak{z}^*_\mathbb{Z} \otimes\mathbb{Q}$,
    it follows that $\pi_I$ is a projective crepant resolution using Proposition~\ref{prop:crepant resol and HK quot} and the $\mathbb{R}^+$-action on $M$.
  \end{enumerate}
\end{remark}

\subsubsection{Wall-Chamber Structure and Diffeomorphisms}

Using these general crepant resolutions, we can observe a natural \textit{wall-chamber structure} on the parameter space. Specifically, the discriminant locus $D'_\mathbb{R}$ determines the wall, and for any parameter $\zeta$ outside this wall, the corresponding hyperkähler quotient $X_\zeta$ is non-singular.

\begin{proposition}\label{prop:smooth ness of Omega'}
  Consider Setup~\ref{setup:HK quot} and Assumption~\ref{assump:standing_assumptions}.
  Let $D'\subset\mathfrak{z}^*_\mathbb{C}$ be the discriminant locus (cf.~Remark~\ref{rem:discriminant locus on z^*_C}), and let $D'_\mathbb{R}\subset\mathfrak{z}^*$ be its real part. Let $\Omega'$ be the complement of the wall $D'_\mathbb{R}\otimes\mathbb{R}^3$ in $\mathfrak{z}^*\otimes\mathbb{R}^3$. Then, for any point $\zeta\in \Omega'$, $X_\zeta$ is non-singular.
\end{proposition}

\begin{proof}
  Under an appropriate hyperkähler rotation, we may assume that $\zeta_I\notin D'_\mathbb{R}$.
  By Lemma~\ref{lemma:discriminant locus and HK moment map}, $X_{(\zeta_I,0,0)}$ is non-singular.
  By Proposition~\ref{prop:crepant resol and HK quot (general case)} and Remark~\ref{remark:to crepant resol prop}, we obtain the following smooth map:
  \begin{equation}
    X_{\zeta} \overset{\pi_K}{\longrightarrow} X_{(\zeta_I,\zeta_J,0)} \overset{\pi_J}{\longrightarrow} X_{(\zeta_I,0,0)}.
  \end{equation}
  Here, $\pi_J$ and $\pi_K$ provide $J$- and $K$-biholomorphic equivalences, respectively.
  Therefore, $X_\zeta$ is non-singular.
\end{proof}

We are now ready to construct the diffeomorphism to the reference space $Y_{\beta_0}$ and evaluate the asymptotic behavior of the metric.
\begin{proposition}\label{prop:preservation of asymp HK met}
  Consider Setup~\ref{setup:HK quot} and Assumption~\ref{assump:standing_assumptions}.
  Fix a projective crepant resolution $Y_{\beta_0}$ ($\beta_0\in \mathfrak{z}^*_\mathbb{Z}$) of the central quotient $X_0$. Then, for any non-singular hyperkähler quotient $X_\zeta$ ($\zeta\in \Omega'$), a diffeomorphism $\Phi_\zeta \colon X_\zeta\to Y_{\beta_0}$ is determined, and the push-forward metric $(\Phi_\zeta)_*g_\zeta$ of the metric $g_\zeta$ is an asymptotic hyperkähler metric on $Y_{\beta_0}$.
\end{proposition}

\begin{proof}
  Fix a projective crepant resolution $Y_{\beta_0}$ ($\beta_0\in \mathfrak{z}^*_\mathbb{Z}$) of the central quotient $X_0$.
  For any $\zeta=(\zeta_I,\zeta_J,\zeta_K)\in \Omega'$, by applying a hyperkähler rotation, we may assume that $\zeta_I \notin D'_\mathbb{R}$.
  We show the following two claims in order:
  \begin{enumerate}[wide=\parindent, itemsep=\medskipamount]
    \item Construction of the diffeomorphism $\Phi_\zeta \colon X_\zeta \to Y_{\beta_0}$:\quad
    Since $X_{(\zeta_I,0,0)}$ is non-singular, similarly to the proof of Proposition~\ref{prop:smooth ness of Omega'}, we obtain the following diffeomorphism:
  \begin{equation}
    F_{\zeta_I} \colon X_{\zeta} \overset{\pi_K}{\longrightarrow} X_{(\zeta_I,\zeta_J,0)} \overset{\pi_J}{\longrightarrow} X_{(\zeta_I,0,0)}.\label{eq:F_zeta_I}
  \end{equation}
  Also, we define a diffeomorphism $\Phi_{\zeta_I,\beta_0}$ between $X_{(\zeta_I,0,0)}$ and $X_{(\beta_0,0,0)}\simeq Y_{\beta_0}$ such that the following diagram commutes:
  \begin{equation}
  \vcenter{
  \xymatrix{
    X_{(\zeta_I,\beta_0,0)}\ar[d]_{\text{HK rot.}} \ar[r]^{\pi_J} & X_{(\zeta_I,0,0)}\ar[d]^{\Phi_{\zeta_I,\beta_0}} \\
    X_{(\beta_0,\zeta_I,0)}\ar[r]_{\pi_J} & X_{(\beta_0,0,0)}
    }}.\label{eq:Phi_{beta,beta'}}
  \end{equation}
  Here, the map $\text{HK rot.} \colon X_{(\zeta_I,\beta_0,0)}\to X_{(\beta_0,\zeta_I,0)}$ represents a hyperkähler rotation.
  Then, we define the diffeomorphism $\Phi_\zeta \colon X_\zeta \to Y_{\beta_0}$ as follows:
  \begin{equation}
    \Phi_\zeta \colon X_{\zeta} \overset{F_{\zeta_I}}{\longrightarrow}X_{(\zeta_I,0,0)} \overset{\Phi_{\zeta_I,\beta_0}}{\longrightarrow} X_{(\beta_0,0,0)}\simeq Y_{\beta_0}.
  \end{equation}

    \item Proof of the asymptotic behavior of the push-forward metric $(\Phi_\zeta)_*g_\zeta$:\quad
  Let $g_\zeta$ be the natural hyperkähler metric on $X_\zeta$.
  Since the moment map $\mu$ satisfies $r^*\mu = r^2 \mu$ with respect to the natural $\mathbb{R}^+$-action on $M$, the following equality holds by Remark~\ref{rem:natural metric on HK quot}:
  \begin{equation}
    r^*g_\zeta = r^2 g_{r^{-2}\zeta}.
  \end{equation}
  Here, we are considering the real scalar multiplication $r \colon X_{r^{-2}\zeta}\to X_{\zeta}$ induced by the scalar multiplication $r \colon \mu^{-1}(r^{-2}\zeta) \to \mu^{-1}(\zeta)$ on $M$. Furthermore, since the natural maps $\pi_J$, $\pi_K$, and hyperkähler rotations are compatible with this real scalar multiplication, we see that their compositions $F_{\zeta_I}$ and $\Phi_{\zeta_I,\beta_0}$ satisfy the following compatibilities:
  \begin{align}
    r^*\circ F_{\zeta_I} &= F_{r^{-2}\zeta_I} \circ r^*,\\
    r^*\circ \Phi_{\zeta_I,\beta_0} &= \Phi_{r^{-2}\zeta_I,r^{-2}\beta_0} \circ r^*.
  \end{align}
  From the above three equalities and the fact that $g_0$ is a cone metric, the following equality holds: for any $r>0$ and $x\in (X_0)_{\mathrm{reg}}\subset Y_{\beta_0}$,
  \begin{align}
    \|(\Phi_\zeta)_*g_\zeta-g_0\|_{g_0(r\cdot x)} &= \|r^*((\Phi_\zeta)_*g_\zeta-g_0)\|_{r^*g_0(x)}\nonumber\\
    &=\|(\Phi_{r^{-2}\zeta_I,r^{-2}\beta_0}\circ F_{r^{-2}\zeta_I})_* g_{r^{-2}\zeta}-g_0\|_{g_0(x)}.
  \end{align}
  Here, since $\Phi_{0,0}=\mathrm{id}$ and $F_{0}=\mathrm{id}$, we obtain the following relation:
  \begin{equation}
    \lim_{r\to \infty} \|(\Phi_\zeta)_*g_\zeta-g_0\|_{g_0(r\cdot x)}=0.
  \end{equation}
  Therefore, it follows that the push-forward metric $(\Phi_\zeta)_*g_\zeta$ is an asymptotic hyperkähler metric on $Y_{\beta_0}$.
  Consequently, the assertion is proved.
  \end{enumerate}
\end{proof}

\begin{remark}\label{rem:to preservation of asymp HK}
  \begin{enumerate}[itemsep=\medskipamount]
    \item For points $\zeta$ on $(D'_\mathbb{R})^c \times\mathfrak{z}^*_\mathbb{C} \subset \Omega'$, the map $\Phi_\zeta \colon X_\zeta \to Y_{\beta_0}$ gives a diffeomorphism. Note that the family of diffeomorphisms $\{\Phi_\zeta\}_\zeta$ can be chosen to depend smoothly on the parameter $\zeta$.

    \item For any rational point $\beta \in \mathfrak{z}^*_\mathbb{Q}\setminus D'_\mathbb{R}$, we have a period-surjective Poisson deformation $\mathcal{Y}^{\beta}\to\mathfrak{z}_\mathbb{C}^*$ of $Y_{\beta}$ (cf.~Prop.~\ref{prop:PD for rational point}). Using its natural smooth trivialization $F_\beta\colon\mathcal{Y}^\beta \to Y_\beta \times\mathfrak{z}^*_\mathbb{C}$ and the diffeomorphism $\Phi_{\beta,\beta_0} \colon Y_\beta \to Y_{\beta_0}$, we obtain a smooth trivialization as follows:
  \begin{equation}
    \mathcal{Y}^{\beta}\overset{F_\beta}{\longrightarrow} Y_{\beta}\times\mathfrak{z}^*_\mathbb{C} \xrightarrow{\Phi_{\beta,\beta_0}\times \mathrm{id}}Y_{\beta_0}\times\mathfrak{z}^*_\mathbb{C}. \label{eq:trivialization}
  \end{equation}

  \end{enumerate}
\end{remark}

\subsection{Proof of the Main Theorem}\label{subsec:Pf of Main thm}

To conclude this section, we prove the Torelli-type theorem (i.e.,~the bijectivity of the period map) for the moduli space of algebraic asymptotic hyperkähler structures on hyperkähler quotients under Assumption~\ref{assump:standing_assumptions}.

\subsubsection{Surjectivity of the Period Map}
The key to the proof lies in showing the existence of a moment map parameter $\zeta$ realizing the given period $\xi$, by utilizing the correspondence between hyperkähler quotients and Poisson deformations prepared in the previous subsections. \medskip

\begin{theorem}\label{main thm:Torelli-type thm for HK quot}
  Consider Setup~\ref{setup:HK quot}, and let $g_0$ be the natural hyperkähler cone metric on the regular locus $(X_0)_{\mathrm{reg}}$ of the central quotient $X_0 \coloneqq \mu^{-1}(0,0,0)/G$ (cf.~Prop.~\ref{prop:cone met on central HK quot}).
  Assume the following two conditions~(cf.~Assumption~\ref{assump:standing_assumptions}):
  \begin{enumerate}
      \item The stable locus $\mu^{-1}_\mathbb{C}(0)^s$ is non-empty.
      \item There exists a generic lattice point $\beta_0\in\mathfrak{z}^*_\mathbb{Z}$ such that the group $G_\mathbb{C}$ acts freely on the stable locus $M^{\beta_0-s}$, and the Kirwan map $\kappa_{\beta_0}$ is surjective.
  \end{enumerate}

  Then, the hyperkähler quotient $Y_{\beta_0}\coloneqq\mu^{-1}(\beta_0,0,0)/G \simeq\mu_\mathbb{C}^{-1}(0)/\!/_{\beta_0} G_\mathbb{C}$ is a projective crepant resolution of $X_0$.
  Consequently, one can consider the moduli space $\mathcal{M}$ of algebraic hyperkähler structures on $Y_{\beta_0}$ asymptotic to $g_0$ at infinity:
  \[ \mathcal{M}=\{(Y_{\beta_0},g,I,J,K)\mid \text{$g$: algebraic hyperkähler metric asymptotic to $g_0$} \}/(\text{isomorphism}). \]
  For this moduli space, the period map
  \[ p \colon \mathcal{M}\to H^2(Y_{\beta_0};\mathbb{R})\otimes\mathbb{R}^3 \]
  (cf.~Def.~\ref{def:period map}) is a bijection onto the period domain $\Omega \coloneqq (D_\mathbb{R}\otimes\mathbb{R}^3)^c$, where $D$ is the wall in $H^2(Y_{\beta_0};\mathbb{C})$ (cf.~Def.~\ref{def:wall}).
\end{theorem}

\begin{proof}
  By Corollary~\ref{cor:crepant_resol_Y_beta}, $Y_{\beta_0}$ is a projective crepant resolution of the central quotient $X_0$.
  Since the injectivity of the period map $p$ and the inclusion $p(\mathcal{M}) \subset \Omega$ have already been established in \S \ref{sec:period map and period domain}, it remains to prove its surjectivity.
  Let $D'$ be the discriminant locus on $\mathfrak{z}^*_\mathbb{C}$ (cf.~Remark~\ref{rem:discriminant locus on z^*_C}), and let its real part be $D'_\mathbb{R} \subset \mathfrak{z}^*$. We show the following three claims in order.
  \medskip

  \begin{enumerate}[wide=\parindent, itemsep=\medskipamount]
    \item Construction of the smooth map $\hat{p} \colon \Omega^{IJ}\to \Omega$:\quad
    Consider the set $\Omega^{IJ} \coloneqq (D'_\mathbb{R})^c\times (D'_\mathbb{R})^c\times\mathfrak{z}^*$, which consists of points on $\mathfrak{z}^*\otimes\mathbb{R}^3$ whose $I$- and $J$-components are not contained in $D'_\mathbb{R}$.
    By Proposition~\ref{prop:preservation of asymp HK met}, for any $\zeta=(\zeta_I,\zeta_J,\zeta_K) \in \Omega^{IJ}$, we have a diffeomorphism $\Phi_{\zeta} \colon X_\zeta \to Y_{\beta_0}$.
    Let $\omega_A$ ($A\in\{I,J,K\}$) be the associated Kähler form for the hyperkähler metric $g_\zeta$ on $X_\zeta$, and let the push-forward forms by the map $\Phi_\zeta$ be $\omega'_A \coloneqq (\Phi_\zeta)_*{\omega_A}$.
    Then, the map $\hat{p} \colon \Omega^{IJ} \to \Omega$ is defined by $\hat{p}(\zeta)=([\omega'_I],[\omega'_J],[\omega'_K])$.
    Here, $[\omega_A']\in H^2(Y_{\beta_0};\mathbb{R})$ denotes the period of $\omega_A'$.
    By Remark~\ref{rem:to preservation of asymp HK}, since the family of maps $\{\Phi_\zeta\}_\zeta$ depends smoothly on the parameter $\zeta$, the map $\hat{p}$ is a smooth map.

    \item Periods on the rational set $\Omega_\mathbb{Q}^{IJ}$:\quad
    Let $\mathfrak{z}^*_\mathbb{Z}$ be the character lattice, and denote its scalar extension to $\mathbb{Q}$ by $\mathfrak{z}^*_\mathbb{Q} = \mathfrak{z}^*_\mathbb{Z}\otimes_\mathbb{Z}\mathbb{Q}$.
    Let $\Omega_\mathbb{Q}^{IJ} \coloneqq \Omega^{IJ}\cap(\mathfrak{z}^*_\mathbb{Q}\otimes\mathbb{Q}^3)$ be the set of rational points in $\Omega^{IJ}$.
    We show that the restriction of the map $\hat{p}$ to the rational set $\Omega_\mathbb{Q}^{IJ}$ can be written as follows:
    \begin{equation}
      \hat{p}|_{\Omega_\mathbb{Q}^{IJ}}=(\kappa_{\beta_0,\mathbb{R}}\otimes\mathbb{R}^3)|_{\Omega_\mathbb{Q}^{IJ}}. \label{eq:for hat p}
    \end{equation}
    Here, the map $\kappa_{\beta_0,\mathbb{R}} \colon \mathfrak{z}^*_\mathbb{Q} \to H^2(Y_{\beta_0};\mathbb{R})$ represents the scalar extension to $\mathbb{R}$ of the Kirwan map $\kappa_{\beta_0,\mathbb{Z}}$, i.e., $\kappa_{\beta_0,\mathbb{R}}=\kappa_{\beta_0,\mathbb{Z}}\otimes_\mathbb{Z}\mathbb{R}$.
    For any $\zeta=(\beta_I,\beta_J,\beta_K)\in \Omega_\mathbb{Q}^{IJ}$, since $\beta_I\notin D'_\mathbb{R}$, we obtain a period-surjective Poisson deformation $\mathcal{Y}^{\beta_I}$ of the projective crepant resolution $Y_{\beta_I}$ of the central quotient $X_0$ (cf.~Prop.\ref{prop:PD for rational point}). Considering the period map with respect to the smooth trivialization \eqref{eq:trivialization} in Remark~\ref{rem:to preservation of asymp HK}, it is naturally isomorphic to the Kirwan map $\kappa_{\beta_0}=\kappa_{\beta_0,\mathbb{Z}}\otimes_\mathbb{Z}\mathbb{C}$.
    Therefore, the period of the point $\alpha \coloneqq \beta_J+i\beta_K\in\mathfrak{z}^*_\mathbb{C}$ is given by $\kappa_{\beta_0}(\alpha)$.
    Applying a hyperkähler rotation to the point $\zeta$ to obtain the point $\zeta'=(\beta_J,\beta_K,\beta_I)$, and similarly considering a period-surjective Poisson deformation $\mathcal{Y}^{\beta_J}$ of the projective crepant resolution $Y_{\beta_J}$, the period of the point $\alpha' \coloneqq \beta_K+i\beta_I$ is given by $\kappa_{\beta_0}(\alpha')$.
    Consequently, we obtain the following relation:
    \[\hat{p}(\zeta)=(\kappa_{\beta_0,\mathbb{R}}\otimes\mathbb{R}^3)(\zeta).\]
    Thus, relation \eqref{eq:for hat p} is established.

    \item Proof of the surjectivity of the period map $p$:\quad
    $\Omega_\mathbb{Q}^{IJ}$ is a dense subset of $\Omega^{IJ}$.
    Also, the map $\hat{p} \colon \Omega^{IJ}\to \Omega$ is continuous, and the target space $\Omega$ is a Hausdorff space.
    Therefore, relation \eqref{eq:for hat p} extends to a relation on $\Omega^{IJ}$:
    \begin{equation}
      \hat{p}=\kappa_{\beta_0,\mathbb{R}}\otimes\mathbb{R}^3. \label{eq:hat p = 3 kirwan maps}
    \end{equation}
    Here, by assumption, the Kirwan map $\kappa_{\beta_0,\mathbb{R}} \colon \mathfrak{z}^* \to H^2(Y_{\beta_0};\mathbb{R})$ is surjective. Furthermore, Proposition~\ref{prop:meaning of wall} on the universal Poisson deformation implies that the discriminant locus $D'$ on $\mathfrak{z}^*_\mathbb{C}$ coincides with the inverse image of the wall $D$ under the Kirwan map $\kappa_{\beta_0}$. In particular, taking the real parts, the following relation is obtained:
    \begin{equation}
      D'_\mathbb{R}=\kappa_{\beta_0,\mathbb{R}}^{-1}(D_\mathbb{R}) \label{eq:rel of D' and D}
    \end{equation}
    Take an arbitrary period $\xi=(\xi_I,\xi_J,\xi_K)\in \Omega$. By applying a hyperkähler rotation if necessary, we may assume $\xi_I,\xi_J \notin D_\mathbb{R}$. Then, by relations \eqref{eq:hat p = 3 kirwan maps} and \eqref{eq:rel of D' and D}, there exists $\zeta\in \Omega^{IJ}$ satisfying $\hat{p}(\zeta)=\xi$.
    By Proposition~\ref{prop:preservation of asymp HK met}, taking the push-forward of the natural hyperkähler metric $g_\zeta$ on the hyperkähler quotient $X_\zeta$ by the diffeomorphism $\Phi_\zeta \colon X_\zeta \to Y_{\beta_0}$, one can regard $g_\zeta$ as an asymptotic hyperkähler metric on $Y_{\beta_0}$.
    Moreover, by the argument on hyperkähler-quotient twistor spaces, it can be shown that $g_\zeta$ is an algebraic asymptotic hyperkähler metric on $Y_{\beta_0}$ (we provide the detailed proof of this algebraicity in Lemma~\ref{lemma:g_zeta is algebraic} below).
    That is, $[X_\zeta \simeq Y_{\beta_0},g_\zeta,I,J,K]\in \mathcal{M}$. From $\hat{p}(\zeta)=\xi$, we obtain:
    \[p([X_\zeta \simeq Y_{\beta_0},g_\zeta,I,J,K]) = \xi.\]
    This implies the surjectivity of the period map $p$. Consequently, the assertion of the theorem follows.
  \end{enumerate}
\end{proof}

\subsubsection{Algebraicity via Twistor Spaces}
To complete the proof of the main theorem, we now verify the algebraicity condition for the hyperkähler structures via their associated twistor spaces.

\begin{lemma}\label{lemma:g_zeta is algebraic}
  The natural hyperkähler metric $g_\zeta$ on a smooth hyperkähler quotient $X_\zeta \simeq Y_{\beta_0}\ (\zeta \in \Omega^{IJ})$ can be regarded as an algebraic asymptotic hyperkähler metric on $Y_{\beta_0}$ (cf.~Def.~\ref{def:alg asymp HK on Y}).
\end{lemma}

\begin{proof}
  By Proposition~\ref{prop:preservation of asymp HK met}, taking the push-forward of the natural hyperkähler metric $g_\zeta$ on the hyperkähler quotient $X_\zeta$ by the diffeomorphism $\Phi_\zeta \colon X_\zeta \to Y_{\beta_0}$, one can regard $g_\zeta$ as an asymptotic hyperkähler metric on $Y_{\beta_0}$. Thus, condition (1) of Definition~\ref{def:alg asymp HK on Y} is satisfied. In the following, we verify conditions (2) and (3) of Definition~\ref{def:alg asymp HK on Y}. We show the following three claims in order.

  \begin{enumerate}[wide=\parindent, itemsep=\medskipamount]
    \item \textit{Construction of the twistor space via GIT quotients}:\quad
      Define the holomorphic moment map $\hat{\mu} \colon Z_{\mathrm{std}}\to\mathfrak{g}^*_\mathbb{C}(2)$ on the standard twistor space $Z_{\mathrm{std}}=\mathcal{O}(1)^{\oplus 2n}$ of $M=\mathbb{H}^n$ by $\hat{\mu}(u)=\mu_\mathbb{C} - 2u\mu_I -u^2\bar{\mu}_\mathbb{C}$ (cf.~{\cite[p.~560]{HKLR}}).
      Denoting the real section on $\mathbb{P}^1$ corresponding to the point $\zeta$ by $s_\zeta(u) = \zeta_\mathbb{C} - 2u\zeta_I - u^2\bar{\zeta}_\mathbb{C}$, the pullback $\tilde{Z}_\zeta \coloneqq \hat{\mu}^{-1}(s_\zeta) \to\mathbb{P}^1$ of $\hat{\mu}$ is determined.
      Here, the Lie group $G_\mathbb{C}$ holomorphically acts on $\tilde{Z}_\zeta$ fiberwise.\medskip

      Suppose that $\zeta\in \Omega^{IJ}$ and $\zeta_I=\beta \in \mathfrak{z}^*_\mathbb{Z}$.
      Using the character corresponding to the lattice point $\beta$, we can naturally construct the relative GIT quotient over $\mathbb{P}^1$:
      \[ Z_{\zeta}^\beta \coloneqq \tilde{Z}_\zeta /\!/_\beta G_\mathbb{C} \longrightarrow \mathbb{P}^1. \]
      By the construction in \cite[p.~560]{HKLR}, this relative GIT quotient $Z_{\zeta}^\beta$ is naturally identified with the twistor space associated with the hyperkähler structure $(X_\zeta,g_\zeta,I,J,K)$.
      (Strictly speaking, a twistor space corresponds to a hyperkähler structure, but for simplicity, we will hereafter refer to it as the twistor space corresponding to the metric $g_\zeta$.)\medskip

    Note that we naturally have
    \[Z_{\zeta}^\beta \subset \mathcal{Y}^\beta(1),\]
    where $\mathcal{Y}^\beta(1)$ denotes the fiber bundle obtained by applying the $\mathbb{C}^*$-gluing construction (cf.~\cite[\S3.2]{Kotani_26_PTM}) to the Poisson deformation $\mathcal{Y}^\beta$. This specific embedding plays a crucial role in establishing the algebraicity of the twistor space.

    \item \textit{Algebraicity of the twistor space for general $\zeta$}:\quad
      Note that for any lattice point $\beta'\in\mathfrak{z}^*_\mathbb{Z}$ belonging to the same chamber as $\zeta_I=\beta$, Lemma~\ref{lemma:same_chamber_PD} yields that their corresponding Poisson deformations coincide: $\mathcal{Y}^\beta=\mathcal{Y}^{\beta'}$. Hence, the twistor space can also be written as:
      \begin{equation}
        Z_{\zeta}^\beta = Z_{\zeta}^{\beta'}\label{eq:unchange of tw sp}
      \end{equation}
      When $\zeta_I \in \mathfrak{z}^*_\mathbb{Q}$, we fix a lattice point $\beta \in \mathfrak{z}^*_\mathbb{Z}$ belonging to the same chamber as $\zeta_I \in \mathfrak{z}^* \setminus D'_\mathbb{R}$. By relation \eqref{eq:unchange of tw sp} and utilizing the $\mathbb{R}^+$-action on $M = \mathbb{H}^n$, we can see that the twistor space corresponding to the hyperkähler metric $g_\zeta$ can still be written in the form $Z_{\zeta}^\beta$.\smallskip

      Consequently, even when $\zeta_I \in \mathfrak{z}^*$ is a general point, the continuity of the real section $s_\zeta$ with respect to $\zeta$ implies that the twistor space corresponding to the metric $g_\zeta$ is given by $Z_{\zeta}^\beta$. This establishes that the twistor space is algebraic, thereby satisfying condition (2) of Definition~\ref{def:alg asymp HK on Y}.

\item \textit{Connection to $Y_{\beta_0}$ via Poisson deformations}:\quad
      Let $Z^\beta_\zeta \to\mathbb{P}^1$ be the twistor space corresponding to the hyperkähler metric $g_\zeta$.
      By the construction of $Z^\beta_\zeta$, its central fiber (i.e., the fiber over the point $0\in\mathbb{P}^1$) can be written as the following holomorphic symplectic variety:
      \[(Z^\beta_\zeta)_0 = \mu_\mathbb{C}^{-1}(\zeta_\mathbb{C})/\!/_\beta G_\mathbb{C} = \mathcal{Y}^\beta_{\zeta_\mathbb{C}}.\]
      Since $\zeta_J \notin D'_\mathbb{R}$, we specifically have $\zeta_\mathbb{C} = \zeta_J+i\zeta_K \notin D'$. By the definition of the discriminant locus $D'$, this implies that the fiber $(Z^\beta_\zeta)_0=\mathcal{Y}^\beta_{\zeta_\mathbb{C}}$ is an affine variety.
      Consequently, the stability condition becomes trivial, yielding
      \[\mathcal{Y}^\beta_{\zeta_\mathbb{C}}=\mu_\mathbb{C}^{-1}(\zeta_\mathbb{C})/\!/_\beta G_\mathbb{C} = \mu_\mathbb{C}^{-1}(\zeta_\mathbb{C})/\!/ G_\mathbb{C} = \mu_\mathbb{C}^{-1}(\zeta_\mathbb{C})/\!/_{\beta_0} G_\mathbb{C}=\mathcal{Y}^{\beta_0}_{\zeta_\mathbb{C}}.\]
      This allows us to write $(Z^\beta_\zeta)_0 = \mathcal{Y}^{\beta_0}_{\zeta_\mathbb{C}}$.
      Combining this with the fact that $\mathcal{Y}^{\beta_0}_0 = Y_{\beta_0}$, we conclude that the central fiber $(Z^\beta_\zeta)_0$ of the twistor space is connected to $Y_{\beta_0}$ via a Poisson deformation.
      Thus, condition (3) of Definition~\ref{def:alg asymp HK on Y} is satisfied, proving that the hyperkähler metric $g_\zeta$ is indeed an algebraic asymptotic hyperkähler metric on $Y_{\beta_0}$.

  \end{enumerate}
\end{proof}

\subsubsection{Twistor-Geometric Reformulation}
  From the perspective of twistor spaces, our main theorem can be reformulated using the principal twistor model \cite{Kotani_26_PTM}.
\begin{remark}
  Consider the principal twistor model $\mathcal{Y}(1)\to\mathcal{C}(2)$ (cf.~\S\ref{subsec:PTM}) determined by the crepant resolution $Y_{\beta_0}$ and the cone metric $g_0$. Here, the base space is $\mathcal{C}\simeq H^2(Y_{\beta_0};\mathbb{C})$. Let $Z_s$ be the twistor model determined by a real section $s$ of $\mathcal{C}(2)$. Then, the following holds:
  \[\text{$Z_s$ is a twistor space} \iff s \not\subset D.\]
  Combined with Proposition~\ref{prop:meaning of wall}, this means that $Z_s$ is a twistor space if and only if its generic fiber is affine. (In particular, under the assumptions of Theorem~\ref{main thm:Torelli-type thm for HK quot}, this provides an answer to \cite[Question~4.30]{Kotani_26_PTM}.)
\end{remark}

%% file: source/Examples.tex
\section{Examples}\label{sec:examples}

In this section, we show that the Torelli-type theorem (Theorem~\ref{main thm:Torelli-type thm for HK quot}) is applicable to toric hyperkähler varieties, Nakajima quiver varieties, and ALE gravitational instantons.
\medskip

To apply Theorem~\ref{main thm:Torelli-type thm for HK quot} to specific hyperkähler quotients, it is necessary to verify Assumption~\ref{assump:standing_assumptions}, which consists of the following three conditions:
\begin{enumerate}
  \item The stable locus $\mu^{-1}_\mathbb{C}(0)^s$ is non-empty.
  \item There exists a generic lattice point $\beta \in \mathfrak{z}^*_\mathbb{Z}$ such that the group $G_\mathbb{C}$ acts freely on the stable locus $M^{\beta-s}$ (cf.~\S \ref{subsec:GIT quot}).
  \item For the generic lattice point $\beta$ above, the Kirwan map $\kappa_{\beta}$ for the hyperkähler quotient $Y_{\beta} \coloneqq \mu^{-1}(\beta,0,0)/G \simeq \mu_\mathbb{C}^{-1}(0)/\!/_{\beta} G_\mathbb{C}$ is surjective (cf.~\S \ref{subsec:PD of HK quot}).
\end{enumerate}

\subsection{Toric Hyperkähler Varieties}
As a standard example to which the Torelli-type theorem is applicable, we introduce the toric hyperkähler variety $X_\zeta(A)$ determined by a coloop-free unimodular matrix $A$.
\medskip

\begin{definition}[{\cite{Bielawski-Dancer00}}]\label{def:toric HK mfd}
  Let $V=\mathbb{C}^n$ be a complex $n$-dimensional vector space.
  Let the standard $n$-dimensional real torus $T^n$ act on $V$ diagonally with weight $1$.
  For an integer $d$ such that $0 < d < n$, consider the action of a $d$-dimensional subtorus $G = T^d \subset T^n$. This inclusion is described via a surjective map $A \colon \mathbb{Z}^n \to \mathbb{Z}^d$ determined by a $d \times n$ integer matrix $A$ (the weight matrix) of rank $d$.
  We also consider the dual representation of $G$ on $V^*$, and consider the representation of $G$ on $M=V\oplus V^*$.
  This representation determines a natural hyperkähler moment map $\mu_A \colon M \to \mathfrak{g}^* \otimes \mathbb{R}^3 \cong \mathbb{R}^d \otimes \mathbb{R}^3$.
  For each parameter $\zeta \in\mathfrak{z}^*\otimes\mathbb{R}^3 \simeq\mathbb{R}^{d}\otimes\mathbb{R}^3$, we call the hyperkähler quotient $X_\zeta(A)$ determined by the hyperkähler moment map $\mu_A$ a \textit{toric hyperkähler variety}.
\end{definition}

\begin{remark}
  If $G$ acts freely on the non-empty fiber $\mu_A^{-1}(\zeta)$, then the real dimension of $X_\zeta(A)$ is $4(n-d)$.
\end{remark}

We prove the following lemma on the condition (1):

\begin{lemma}\label{lemma:for coloop-free}
  Assume that the weight matrix $A$ has no coloops (i.e.,~every column vector of $A$ is contained in the subspace spanned by the remaining column vectors).
  Then, $\mu_\mathbb{C}^{-1}(\mathbf{0})^s \neq \emptyset$.
\end{lemma}

\begin{proof}
  Assume that every column vector of the weight matrix $A$ is contained in the subspace spanned by the remaining column vectors.
  Then, the kernel $\ker(A) \subset \mathbb{C}^n$ of $A$ is not contained in any coordinate hyperplane $\{x_i = 0\}$, so one can choose a vector $\mathbf{x} = (x_1, \ldots, x_n) \in \ker(A)\subset\mathbb{C}^n$ whose components are all non-zero.
  Here, we define a point $\mathbf{p} = (\mathbf{z},\mathbf{w})$ in the representation space $M = V \oplus V^*$ by $z_i = w_i = \sqrt{x_i}$. Then, for the complex moment map, we have $\mu_\mathbb{C}(\mathbf{p}) = \sum_{i=1}^n  z_i w_i \mathbf{A}_i= \sum_{i=1}^n  x_i \mathbf{A}_i = \mathbf{0}$. Here, $\mathbf{A}_i$ represents the $i$-th $d$-dimensional column vector of the weight matrix $A$.
  \medskip

  We show that this point $\mathbf{p}$ is a stable point by the Hilbert--Mumford criterion (cf.~Prop.~\ref{prop:Hilbert-Munford criterion}).
  Any one-parameter subgroup $\lambda_{\mathbf{v}} \colon \mathbb{C}^* \to G_\mathbb{C}$ of the group $G_\mathbb{C} \cong (\mathbb{C}^*)^d$ can be written by a certain vector $\mathbf{v} \in \mathbb{Z}^d$ as follows: for a point $\mathbf{p}=(\mathbf{z},\mathbf{w})\in M$,
  \begin{equation}
    \lambda_{\mathbf{v}}(t) \cdot (\mathbf{z},\mathbf{w}) = ( t^{\langle \mathbf{v}, \mathbf{A}_1 \rangle} z_1, \dots, t^{\langle \mathbf{v}, \mathbf{A}_n \rangle} z_n, \ t^{-\langle \mathbf{v}, \mathbf{A}_1 \rangle} w_1, \dots, t^{-\langle \mathbf{v}, \mathbf{A}_n \rangle} w_n ).
  \end{equation}

  Since $z_i \neq 0$ and $w_i \neq 0$ for all $i$, for the limit $\displaystyle\lim_{t \to 0} \lambda_{\mathbf{v}}(t) \cdot \mathbf{p}$ to exist in $M$, the conditions $\langle \mathbf{v}, \mathbf{A}_i \rangle \ge 0$ and $-\langle \mathbf{v}, \mathbf{A}_i \rangle \ge 0$ must hold simultaneously.
  This implies $\langle \mathbf{v}, \mathbf{A}_i \rangle = 0$ for any $i$, but since $A$ is surjective, we deduce $\mathbf{v} = \mathbf{0}$.
  Therefore, the limit does not exist for any non-trivial one-parameter subgroup. Consequently, the point $\mathbf{p}\in M$ forms a stable point, which proves $\mu_\mathbb{C}^{-1}(\mathbf{0})^s \neq \emptyset$.
\end{proof}

We prove the following lemma on the condition (2):

\begin{lemma}\label{lemma:free_action_toric}
  Assume that the weight matrix $A$ is unimodular (i.e.,~every non-zero $d \times d$ minor is $\pm 1$).
  Then, for a generic lattice point $\beta \in \mathfrak{z}^*_\mathbb{Z}\simeq\mathbb{Z}^d$, the group $G_\mathbb{C}$ acts freely on the stable locus $M^{\beta-s}$.
\end{lemma}

\begin{proof}
  Let $\mathbf{p} = (\mathbf{z},\mathbf{w}) \in M^{\beta-s}$ and let $\mathbf{t} = (t_1, \dots, t_d) \in G_\mathbb{C} \cong (\mathbb{C}^*)^d$ be an element fixing $\mathbf{p}$.
  The action of $\mathbf{t}$ on $\mathbf{p}$ is given by
  \begin{equation}
    \mathbf{t} \cdot (\mathbf{z},\mathbf{w}) = (\mathbf{t}^{\mathbf{A}_1} z_1, \dots, \mathbf{t}^{\mathbf{A}_n} z_n, \ \mathbf{t}^{-\mathbf{A}_1} w_1, \dots, \mathbf{t}^{-\mathbf{A}_n} w_n),
  \end{equation}
  where $\mathbf{t}^{\mathbf{A}_i} \coloneqq t_1^{a_{1,i}} \cdots t_d^{a_{d,i}}$ and $a_{j,i}$ is the $(j,i)$-entry of the matrix $A$.\medskip

  Since $\mathbf{p}$ is a stable point, its stabilizer is finite, which implies that the set of column vectors $\mathbf{A}_i$ corresponding to the non-zero components of $\mathbf{p}$ (i.e., $z_i \neq 0$ or $w_i \neq 0$) spans $\mathbb{R}^d$.
  As $A$ is unimodular, we can choose $d$ such column vectors, say $\{\mathbf{A}_{i_1}, \dots, \mathbf{A}_{i_d}\}$, that form a $\mathbb{Z}$-basis of $\mathbb{Z}^d$.
  Since the condition $\mathbf{t} \cdot \mathbf{p} = \mathbf{p}$ implies that $\mathbf{t}^{\mathbf{A}_{i_k}} = 1$ for all $k = 1, \dots, d$, it follows that $\mathbf{t} = (1, \dots, 1) \in G_\mathbb{C}$.
  Thus, the stabilizer is trivial, which proves that $G_\mathbb{C}$ acts freely on the stable locus.
\end{proof}

Lemma~\ref{lemma:free_action_toric} guarantees that for a generic lattice point $\beta \in \mathfrak{z}^*_\mathbb{Z}\simeq\mathbb{Z}^d$, the toric hyperkähler variety $Y_{\beta}=X_{(\beta,0,0)}(A)$ defines a projective crepant resolution of the central quotient $X_{\mathbf{0}}(A)$ (cf.~Corollary~\ref{cor:crepant_resol_Y_beta}).
Furthermore, the surjectivity (and in particular, the fact that it is an isomorphism) of the Kirwan map of $Y_{\beta}$ (condition (3)) is proved by Konno \cite{Konno00}.
Consequently, we obtain the following proposition:
\begin{proposition}\label{prop:Torelli for toric HK}
  Consider the toric hyperkähler variety $X_\zeta(A)$ determined by the weight matrix $A$, as in Definition~\ref{def:toric HK mfd}.
  If the weight matrix $A$ is a coloop-free unimodular matrix, then for a generic integer vector (lattice point) $\beta\in \mathfrak{z}^*_\mathbb{Z} \simeq \mathbb{Z}^d$, the toric hyperkähler variety $Y_{\beta}=X_{(\beta,0,0)}(A)$ defines a projective crepant resolution of the central quotient $X_{\mathbf{0}}(A)$. Furthermore, the Torelli-type theorem (Theorem~\ref{main thm:Torelli-type thm for HK quot}) holds for algebraic asymptotic hyperkähler structures on $Y_{\beta}$.
\end{proposition}

\subsection{Nakajima Quiver Varieties}
As another fruitful example to which the Torelli-type theorem is applicable, we introduce the Nakajima quiver variety $X_\zeta(\mathbf{v},\mathbf{w})$ whose pair of dimension vectors $(\mathbf{v},\mathbf{w})$ forms a strict Schur root.
\medskip

\begin{definition}[\cite{Nakajima94}]\label{def:Nakajima quiver variety}
  Let $Q_0 = (I, A)$ be a finite quiver with vertex set $I$ and arrow set $A$. We assume that $Q_0$ has no oriented cycles.
  For dimension vectors $\mathbf{v} = (v_k)_{k \in I}$ and $\mathbf{w} = (w_k)_{k \in I}$ in $\mathbb{Z}_{\ge 0}^I$, we set complex vector spaces $V_k = \mathbb{C}^{v_k}$ and $W_k = \mathbb{C}^{w_k}$.
  We define the $W$-framed representation space $\mathbf{V}=\operatorname{Rep}(Q_0;\mathbf{v},\mathbf{w})$ of $Q_0$ by
  \begin{equation}
    \mathbf{V} \coloneqq \bigoplus_{h \in A} \operatorname{Hom}(V_{\operatorname{out}(h)}, V_{\operatorname{in}(h)}) \oplus \bigoplus_{k \in I} \operatorname{Hom}(W_k, V_k).
  \end{equation}
  Consider the hyperkähler vector space $M \coloneqq \mathbf{V} \oplus \mathbf{V}^*$.
  The compact Lie group $G \coloneqq \prod_{k \in I} U(v_k)$ naturally acts on $M$, which determines a hyperkähler moment map $\mu_{(\mathbf{v},\mathbf{w})} = (\mu_\mathbb{R}, \mu_\mathbb{C}) \colon M \to \mathfrak{g}^* \otimes \mathbb{R}^3$.
  For any parameter $\zeta \in \mathfrak{z}^* \otimes \mathbb{R}^3$, the hyperkähler quotient $X_\zeta(\mathbf{v},\mathbf{w}) \coloneqq \mu_{(\mathbf{v},\mathbf{w})}^{-1}(\zeta)/G$ is called the \textit{Nakajima quiver variety} (cf.~{\cite{Nakajima94}}).
\end{definition}

To treat the framing vector $\mathbf{w}$ algebraically, we introduce the \textit{framed quiver} $Q$ of $Q_0$ by adding an extra vertex $\infty$ to $I$ and drawing $w_k$ arrows from $\infty$ to each vertex $k \in I$ (cf.~\cite{CrawleyBoevey01}).
Then, the representation space $\mathbf{V}$ canonically corresponds to the representations of $Q$ with the dimension vector $\alpha = (\mathbf{v}, 1) \in \mathbb{Z}_{\ge 0}^{I \cup \{\infty\}}$.
Consequently, $M$ can be regarded as the representation space of the double quiver $Q^\sharp$ of $Q$ (cf.~\cite[Def.~5.1]{kirillov16quiver}).
We prepare the following definition:
\begin{definition}[cf.~\cite{CrawleyBoevey01}]
  We say that a pair of dimension vectors $(\mathbf{v},\mathbf{w})$ forms a \textit{strict Schur root} if the corresponding dimension vector $\alpha \coloneqq (\mathbf{v}, 1)$ is a strict Schur root of the Kac--Moody Lie algebra $\mathfrak{g}(Q)$ associated with the framed quiver $Q$. Namely, it satisfies the following two conditions:
  \begin{enumerate}
    \item $\alpha$ is a positive root of $\mathfrak{g}(Q)$.
    \item For any non-trivial decomposition $\alpha = \beta_1 +\cdots + \beta_k~(k\geq 2)$ into positive roots of $\mathfrak{g}(Q)$, it satisfies $p(\alpha) > \sum_{i=1}^k p(\beta_i)$. Here, $p(\alpha) \coloneqq 1-\langle\alpha,\alpha\rangle ~\geq 0$, where $\langle -, - \rangle$ is the Euler form of $Q$ (cf.~\cite[\S~1.5]{kirillov16quiver}).
  \end{enumerate}
\end{definition}

\begin{remark}
  In this case, the real dimension of the Nakajima quiver variety $X_\zeta(\mathbf{v},\mathbf{w})$ is $4p(\alpha)>0$.
\end{remark}

Any point $\mathbf{p}=(B,i,j) \in M$ determines a representation $V_{\mathbf{p}}$ of the double quiver $Q^\sharp$ with dimension vector $\alpha=(\mathbf{v},1)$ under the natural identification
\begin{equation}
  M \simeq \bigoplus_{h \in A^\sharp} \operatorname{Hom}(V_{\operatorname{out}(h)}, V_{\operatorname{in}(h)}) \oplus \bigoplus_{k \in I} \operatorname{Hom}(W_k, V_k) \oplus \bigoplus_{k \in I} \operatorname{Hom}(V_k, W_k).
\end{equation}
From this perspective, it is well known that the $\beta$-stability of a point $\mathbf{p}=(B,i,j)$ can be restated in terms of the representation $V_{\mathbf{p}}$ as follows.

\begin{lemma}[cf.~{\cite[Thm.~10.32]{kirillov16quiver}}]\label{lemma:stability for quiver representation} 
  For any lattice point $\beta \in \mathfrak{z}^*_\mathbb{Z} \simeq \mathbb{Z}^I$, the following two conditions are equivalent:
  \begin{enumerate}
    \item The point $\mathbf{p}=(B,i,j)\in M$ is $\beta$-stable.
    \item The representation $V_{\mathbf{p}}$ of $Q^\sharp$ is \textit{$\beta$-stable} (cf.~\cite[Def.~10.31]{kirillov16quiver}).
    That is, for any $B$-invariant proper non-zero subspace $V' \subsetneq V:=\displaystyle\bigoplus_{k\in I}V_k$,
    \begin{align}
      V' \subset \ker(j) &\implies \beta \cdot \mathbf{v}' < 0, \\
      \operatorname{im}(i) \subset V' &\implies \beta \cdot \mathbf{v}' < \beta \cdot \mathbf{v}.
    \end{align}
    Here, $\mathbf{v}' \coloneqq (\dim V'_k)_{k\in I} \in \mathbb{Z}^I$ denotes the dimension vector of $V'$, and $\mathbf{v}$ is the dimension vector of $V$.
  \end{enumerate}
\end{lemma}

From this restatement and the result of \cite{CrawleyBoevey01}, the following lemma on the condition (1) immediately follows:

\begin{lemma}[Crawley-Boevey {\cite[Theorem~1.2]{CrawleyBoevey01}}]\label{lemma:nonempty_stable_quiver}
  Assume that the pair of dimension vectors $(\mathbf{v},\mathbf{w})$ forms a strict Schur root.
  Then, the stable locus $\mu^{-1}_\mathbb{C}(0)^s$ is non-empty.
\end{lemma}

\begin{proof}
  By \cite[Theorem~1.2]{CrawleyBoevey01}, the condition that the dimension vector $\alpha = (\mathbf{v},1)$ is a strict Schur root of $\mathfrak{g}(Q)$ ensures the existence of a simple representation of the associated preprojective algebra $\Pi^0(Q)$.
  This is equivalent to the existence of a point $\mathbf{p} \in \mu_\mathbb{C}^{-1}(\mathbf{0})$ such that the corresponding representation $V_{\mathbf{p}}$ of $Q^\sharp$ is a simple representation.
  Therefore, by Lemma~\ref{lemma:stability for quiver representation}, this point $\mathbf{p}$ is a stable point (since a simple representation has no proper subrepresentations), which proves $\mu_\mathbb{C}^{-1}(\mathbf{0})^s \neq \emptyset$.
\end{proof}

Using the restatement in Lemma~\ref{lemma:stability for quiver representation},
we prove the following lemma on the condition (2):

\begin{lemma}\label{lemma:free_action_quiver}
  For a generic lattice point $\beta \in \mathfrak{z}^*_\mathbb{Z}\simeq\mathbb{Z}^I$, the Lie group $G_\mathbb{C} = \prod_{k \in I} GL(V_k)$ acts freely on the stable locus $M^{\beta-s}$.
\end{lemma}

\begin{proof}
  Let $\mathbf{p} = (B, i, j) \in M^{\beta-s}$ be a stable point, and let $g = (g_k)_{k \in I} \in G_\mathbb{C}$ be an element fixing $\mathbf{p}$.
  We consider the linear map $\phi \coloneqq g - \operatorname{id}_V \colon V \to V$, defined component-wise by $\phi_k = g_k - \operatorname{id}_{V_k}$.
  The condition $g \cdot \mathbf{p} = \mathbf{p}$ means $g_{\operatorname{in}(h)} B_h g_{\operatorname{out}(h)}^{-1}= B_h $, $g_k i_k = i_k$, and $j_k g_k = j_k$.
  By subtracting the identity, we obtain
  \begin{equation}
    \phi_{\operatorname{in}(h)} B_h = B_h \phi_{\operatorname{out}(h)}, \quad \phi_k i_k = 0, \quad j_k \phi_k = 0.
  \end{equation}
  The first equation shows that $\ker(\phi)$ and $\operatorname{im}(\phi)$ are $B$-invariant subspaces of $V$.
  The second and third equations imply $\operatorname{im}(i) \subset \ker(\phi)$ and $\operatorname{im}(\phi) \subset \ker(j)$, respectively.\medskip

  Assume for contradiction that $g \neq \operatorname{id}_V$, which means $\phi \neq 0$.
  Then $\operatorname{im}(\phi)$ is a non-zero subspace, and $\ker(\phi)$ is a proper subspace ($\ker(\phi) \subsetneq V$).
  By applying Lemma~\ref{lemma:stability for quiver representation} to these $B$-invariant subspaces, we obtain two inequalities:
  \begin{align}
    0 \neq \operatorname{im}(\phi) \subset \ker(j) &\implies \beta \cdot \mathbf{v}_{\operatorname{im}(\phi)} < 0, \\
    \operatorname{im}(i) \subset \ker(\phi) \subsetneq V &\implies \beta \cdot \mathbf{v}_{\ker(\phi)} < \beta \cdot \mathbf{v}.
  \end{align}
  Here, $\mathbf{v}_{\ker(\phi)}$ and $\mathbf{v}_{\operatorname{im}(\phi)}$ denote the dimension vectors of $\ker(\phi)$ and $\operatorname{im}(\phi)$, respectively.\medskip

  By the rank-nullity theorem, we have $\dim V_k = \dim \ker(\phi_k) + \dim \operatorname{im}(\phi_k)$ for each $k \in I$, which yields the following equality:
  \begin{equation}
    \beta \cdot \mathbf{v} = \beta \cdot \mathbf{v}_{\ker(\phi)} + \beta \cdot \mathbf{v}_{\operatorname{im}(\phi)}.
  \end{equation}
  However, summing the two inequalities above yields $\beta \cdot \mathbf{v}_{\ker(\phi)} + \beta \cdot \mathbf{v}_{\operatorname{im}(\phi)} < \beta \cdot \mathbf{v}$, which contradicts the equality.
  Therefore, we must have $\phi = 0$, meaning $g = \operatorname{id}_V$. This proves that the $G_\mathbb{C}$-action on $M^{\beta-s}$ is free.
\end{proof}
Lemma~\ref{lemma:free_action_quiver} guarantees that for a generic lattice point $\beta \in \mathfrak{z}^*_\mathbb{Z}\simeq\mathbb{Z}^I$, the Nakajima quiver variety $Y_{\beta}=X_{(\beta,0,0)}(\mathbf{v},\mathbf{w})$ defines a projective crepant resolution of the central quotient $X_{\mathbf{0}}(\mathbf{v},\mathbf{w})$ (cf.~Corollary~\ref{cor:crepant_resol_Y_beta}).
Furthermore, the surjectivity of the Kirwan map of $Y_{\beta}$ (condition (3)) is proved by McGerty--Nevins \cite{McGerty-Nevins18}.
Consequently, we obtain the following proposition.
\begin{proposition}\label{prop:Torelli for Nakajima quiver}
  Let $Q_0$ be a finite quiver without oriented cycles, and let $X_\zeta(\mathbf{v},\mathbf{w})$ be the associated Nakajima quiver variety.
  Assume that the pair of dimension vectors $(\mathbf{v},\mathbf{w})$ forms a strict Schur root for the framed quiver $Q$ of $Q_0$.
  Then, for a generic lattice point $\beta \in \mathfrak{z}^*_\mathbb{Z}\simeq\mathbb{Z}^I$, the Nakajima quiver variety $Y_{\beta}=X_{(\beta,0,0)}(\mathbf{v},\mathbf{w})$ defines a projective crepant resolution of the central quotient $X_{\mathbf{0}}(\mathbf{v},\mathbf{w})$.
  Furthermore, the Torelli-type theorem (Theorem~\ref{main thm:Torelli-type thm for HK quot}) holds for algebraic asymptotic hyperkähler structures on $Y_{\beta}$.
\end{proposition}

We give a typical example of Nakajima quiver varieties whose pair of dimension vectors forms a strict Schur root.

\begin{example}[Moduli of instantons on ALE spaces]
  Let $Q_0$ be an extended Dynkin quiver of type $\tilde{A}, \tilde{D}$, or $\tilde{E}$, and let $\delta$ be its minimal positive imaginary root (cf.~\cite{kirillov16quiver}).
  We consider a framed quiver $Q$ obtained by setting the framing dimension vector $\mathbf{w} = (r, 0, \dots, 0)$ (i.e., adding a framing of rank $r$ only to the extending vertex).
  For a dimension vector $\mathbf{v} = k\delta$ with $k \ge 1$, the extended dimension vector is given by $\alpha = (k\delta, 1)$.
  By a straightforward calculation of the Euler form, one can verify that if $r \ge 2$, then the vector $\alpha = (k\delta, 1)$ is always a strict Schur root of $\mathfrak{g}(Q)$. (On the other hand, if $r=1$, then it is a Schur root but not strict, since $p(k\delta, 1) = p((k-1)\delta, 1) + p(\delta,0)$ holds).\medskip

  In this case, for the generic lattice point $\beta$, the Nakajima quiver variety $Y_{\beta}$ is a smooth hyperkähler manifold, which corresponds to the moduli space of framed instantons (or torsion-free sheaves) of rank $r$ and charge $k$ on the associated ALE space (cf.~\cite{Nakajima94}).
\end{example}

\subsection{ALE Gravitational Instantons}\label{subsec:ALE-grav inst}
As another application, we recall the hyperkähler quotient construction of ALE gravitational instantons due to Kronheimer \cite{Kronheimer89_HK_quot}.
\medskip

For an extended Dynkin quiver $Q_0$ of type $\tilde{A}$, $\tilde{D}$, or $\tilde{E}$, consider the (non-framed) quiver representation space $\mathbf{V} \coloneqq \operatorname{Rep}(Q_0;\delta,\mathbf{0})$ (cf.~Definition~\ref{def:Nakajima quiver variety}) whose dimension vector is the minimal positive imaginary root $\delta$ (cf.~{\cite{kirillov16quiver}}).
In this case, the subgroup $T \coloneqq S^1 \subset G_0 \coloneqq \prod_{k \in I} U(\delta_k)$ consisting of scalar multiplication acts trivially on $M \coloneqq \mathbf{V} \oplus \mathbf{V}^*$.
Therefore, we consider the effective action of the quotient Lie group $G \coloneqq G_0/T$ on the representation space $M$, and the induced hyperkähler moment map $\mu \colon M \to \mathfrak{g}^* \otimes \mathbb{R}^3$.
Then, the hyperkähler quotient corresponds to the ALE gravitational instanton constructed by Kronheimer \cite{Kronheimer89_HK_quot}.
From the results of \cite{Kronheimer89_HK_quot}, one can readily verify that $\mu$ satisfies Assumption~\ref{assump:standing_assumptions} (in particular, the Kirwan map is an isomorphism).\medskip

Consequently, for a generic lattice point $\beta \in \mathfrak{z}^*_\mathbb{Z}$, the projective crepant resolution $Y_{\beta}$ exactly coincides with the minimal resolution $Y$ of the quotient space $\mathbb{C}^2/\Gamma$.
By applying Theorem~\ref{main thm:Torelli-type thm for HK quot} to this minimal resolution $Y$, one obtains an algebro-geometric proof of the Torelli-type theorem for ALE hyperkähler metrics on $Y$, which was originally established by Kronheimer \cite{Kronheimer89_HK_quot, Kronheimer89_Torelli-type} via analytic methods.

%% file: source/Future_Directions.tex
\section{Future Directions}\label{sec:future directions}

In this paper, we have proved the Torelli-type theorem (Theorem~\ref{main thm:Torelli-type thm for HK quot}) under Assumption~\ref{assump:standing_assumptions} in the case where a compact Lie group $G \subset \mathrm{Sp}(n)$ acts linearly on the vector space $\mathbb{H}^n$.
The proof of this theorem relies on the algebraic description of hyperkähler quotients as GIT quotients, using the Kempf--Ness type theorem for affine varieties \cite{King94}.
Therefore, it is natural to consider the following problem:

\begin{problem}
  Consider a class of hyperkähler quotients determined by a hyperkähler moment map $\mu \colon M \to \mathfrak{g}^* \otimes \mathbb{R}^3$, where $M$ is not necessarily the vector space $\mathbb{H}^n$, for which Assumption~\ref{assump:standing_assumptions} holds.
  In such cases, establish a Kempf--Ness type isomorphism and prove the surjectivity of the period map $p$.
\end{problem}

The injectivity of the period map $p$ can be proved without relying on the structure of hyperkähler quotients (cf.~Proposition~\ref{prop:inj of period map}).
From this perspective, we pose the following more general problem:

\begin{problem}
  Prove the surjectivity of the period map $p$ given any projective crepant resolution $Y$ of a general conical symplectic variety $X$, which does not necessarily admit the structure of a hyperkähler quotient.
  One possible approach is to show that the image $p(\mathcal{M}) \subset \Omega$ of the period map forms a non-empty open and closed subset of $\Omega$.
\end{problem}

\begin{remark}
  It is known that when the conical symplectic variety $X$ has isolated singularities, the image $p(\mathcal{M})$ of the period map forms an open set \cite{Kotani_26_PTM}.
\end{remark}

Thus far, our main objects of study have been conical symplectic varieties admitting projective crepant resolutions $Y$.
However, a general conical symplectic variety does not necessarily admit a projective crepant resolution; instead, a partial resolution (a $\mathbb{Q}$-factorial terminalization) $Y$ is obtained via the Minimal Model Program (cf.~\cite{BCHM10}).
In light of this background, we pose the following problem:

\begin{problem}
  Let $X$ be an arbitrary conical symplectic variety, and let $Y$ be its partial resolution.
  In this case, formulate the notion of algebraic asymptotic hyperkähler metrics and define the period map $p$.
  Furthermore, extend the Torelli-type theorem to this setting.
\end{problem}

\begin{remark}
  Since the universal Poisson deformation of $Y$ exists even for partial resolutions (cf.~\cite{Namikawa11}), one can consider the principal twistor model as in \cite{Kotani_26_PTM}. It is expected that at least the injectivity of the period map can be proved using the PTM.
\end{remark}

%% file: source/Acknowledgments.tex
\section*{Acknowledgments}
This work was supported by JST SPRING, Japan Grant Number JPMJSP2180.

%% file: paper2_ver.4.bbl
\begin{thebibliography}{10}

\bibitem{Bielawski-Dancer00}
R.~Bielawski and A.~S. Dancer.
\newblock The geometry and topology of toric hyperk\"{a}hler manifolds.
\newblock {\em\JournalTitle{Commun. Anal. Geom}}, 8(4):727--759, 2000.

\bibitem{BCHM10}
C.~Birkar, P.~Cascini, C.~D. Hacon, and J.~McKernan.
\newblock Existence of minimal models for varieties of log general type.
\newblock {\em\JournalTitle{J. Amer. Math. Soc}}, 23(2):405--468, 2010.

\bibitem{CrawleyBoevey01}
W.~Crawley-Boevey.
\newblock Geometry of the moment map for representations of quivers.
\newblock {\em\JournalTitle{Compos. Math.}}, 126(3):257--293, 2001.

\bibitem{Görtz-Wedhorn10}
U.~G{\"o}rtz and T.~Wedhorn.
\newblock {\em Algebraic Geometry: Part I: Schemes. With Examples and Exercises}.
\newblock Advanced Lectures in Mathematics. Vieweg+Teubner Verlag, 2010.

\bibitem{HKLR}
N.~J. Hitchin, A.~Karlhede, U.~Lindstr\"{o}m, and M.~Ro\v{c}ek.
\newblock Hyperk\"{a}hler metrics and supersymmetry.
\newblock {\em\JournalTitle{Comm. Math. Phys.}}, 108:535--589, 1987.

\bibitem{Kaledin_06}
D.~Kaledin.
\newblock Geometry and topology of symplectic resolutions.
\newblock {\em\JournalTitle{Proc. Sympos. Pure Math.}}, 80, 09 2006.

\bibitem{King94}
A.~D. King.
\newblock Moduli of representations of finite dimensional algebras.
\newblock {\em\JournalTitle{Q. J. Math.}}, 45(4):515--530, 1994.

\bibitem{kirillov16quiver}
J.~Kirillov, Alexander.
\newblock {\em Quiver Representations and Quiver Varieties}, volume 174 of {\em Graduate Studies in Mathematics}.
\newblock American Mathematical Society, 2016.

\bibitem{Konno00}
H.~Konno.
\newblock Cohomology rings of toric hyperk\"ahler manifolds.
\newblock {\em\JournalTitle{Int. J. Math.}}, 11(08):1001--1026, 2000.

\bibitem{Kotani_26_PTM}
R.~Kotani.
\newblock Principal twistor models and asymptotic hyperk\"ahler metrics.
\newblock {\em\JournalTitle{arXiv preprint arXiv:2603.03923}}, 2026.

\bibitem{Kronheimer89_HK_quot}
P.~B. Kronheimer.
\newblock The construction of {ALE} spaces as hyper-{K}\"ahler quotients.
\newblock {\em\JournalTitle{J. Differential Geom.}}, 29(3):665--685, 1989.

\bibitem{Kronheimer89_Torelli-type}
P.~B. Kronheimer.
\newblock A {T}orelli-type theorem for gravitational instantons.
\newblock {\em\JournalTitle{J. Differential Geom.}}, 29(3):685--697, 1989.

\bibitem{Losev12}
I.~Losev.
\newblock Isomorphisms of quantizations via quantization of resolutions.
\newblock {\em\JournalTitle{Adv. Math.}}, 231(3--4):1216--1270, 2012.

\bibitem{McGerty-Nevins18}
K.~McGerty and T.~Nevins.
\newblock {K}irwan surjectivity for quiver varieties.
\newblock {\em\JournalTitle{Invent. math.}}, 212:161--187, 04 2018.

\bibitem{Nagaoka21}
T.~Nagaoka.
\newblock The universal {P}oisson deformation of hypertoric varieties and some classification results.
\newblock {\em\JournalTitle{Pacific J. Math.}}, 313(2):459--508, 2021.

\bibitem{Nakajima94}
H.~Nakajima.
\newblock Instantons on {ALE} spaces, quiver varieties, and {K}ac-{M}oody algebras.
\newblock {\em\JournalTitle{Duke Math. J.}}, 76(2):365--416, 1994.

\bibitem{Namikawa08}
Y.~Namikawa.
\newblock Flops and {P}oisson deformations of symplectic varieties.
\newblock {\em\JournalTitle{Publ. Res. Inst. Math. Sci.}}, 44(2):259--314, 2008.

\bibitem{Namikawa11}
Y.~Namikawa.
\newblock Poisson deformations of affine varieties.
\newblock {\em\JournalTitle{Duke Math. J.}}, 156:51--85, 2011.

\end{thebibliography}
